\documentclass[11pt]{amsart}
\usepackage{amsfonts, amstext, amsmath, amsthm, amscd, amssymb}
\usepackage{epsfig, graphicx, psfrag}
\usepackage{color}
\usepackage{pinlabel}
\usepackage{import}

\usepackage[
pdfborderstyle={},
pdfborder={0 0 0},
pagebackref,
pdftex]{hyperref}

\renewcommand*{\backref}[1]{}
\renewcommand*{\backrefalt}[4]{
  \ifcase #1 %
   [No citations.]%
  \or
   [#2]%
  \else
   [#2]%
  \fi
}

\catcode `\@=11
\long\def\@savemarbox#1#2{\global\setbox#1\vtop{\hsize\marginparwidth 
  \@parboxrestore\tiny\raggedright #2}}
\catcode`\@=12

\newcommand{\RR}{{\mathbb{R}}}

\def\from{\colon\thinspace}

\renewcommand{\setminus}{{\smallsetminus}}

\theoremstyle{plain}
\newtheorem{theorem}{Theorem}[section]
\newtheorem{corollary}[theorem]{Corollary}
\newtheorem{lemma}[theorem]{Lemma}

\newtheorem{problem}[theorem]{Problem}

\newtheorem*{namedtheorem}{\theoremname}
\newcommand{\theoremname}{testing}
\newenvironment{named}[1]{\renewcommand{\theoremname}{#1}\begin{namedtheorem}}{\end{namedtheorem}}

\theoremstyle{definition}
\newtheorem{definition}[theorem]{Definition}
\newtheorem{example}[theorem]{Example}
\newtheorem{remark}[theorem]{Remark}
\newtheorem{terminology}[theorem]{Terminology}

\def\zep{\epsilon}
\def\Hd{H_\Zd}
\def\Dd{D_\Zd}
\def\Hc{H_\Zc}
\def\Dc{D_\Zc}

\def\Bth{{\mathbb B}^3}
\def\Rth{{\mathbb R}^3}
\def\Rt{{\mathbb R}^2}
\def\Sth{{\mathbb S}^3}
\def\St{{\mathbb S}^2}
\def\So{{\mathbb S}^1}
\def\za{\alpha}
\def\zb{\beta}
\def\zc{\gamma}
\def\zd{\delta}
\def\zth{\theta}
\def\Zc{\Gamma}
\def\Zf{\Phi}
\def\Zd{\Delta}
\def\bd{\partial}

\def\R{\mathbb R}
\def\Z{\mathbb Z}
\def\inv{^{-1}}

\title{Bi-twist manifolds and two-bridge knots}

\author{J. W. Cannon}
\address{Department of Mathematics\\ 
Brigham Young University\\ 
Provo, UT 84602\\ U.S.A.} 

\author{W. J. Floyd}
\address{Department of Mathematics\\ 
Virginia Tech\\
Blacksburg, VA 24061\\ U.S.A.} 
\email{floyd@math.vt.edu}
\urladdr{http://www.math.vt.edu/people/floyd}

\author{L. Lambert}

\author{W. R. Parry}
\address{Department of Mathematics\\ 
Eastern Michigan University\\
Ypsilanti, MI 48197\\ U.S.A.} 
\email{wparry@emich.edu}

\author{J. S. Purcell}
\address{Department of Mathematics\\ 
Brigham Young University\\ 
Provo, UT 84602\\ U.S.A.} 
\email{jpurcell@math.byu.edu}
\urladdr{http://math.byu.edu/~jpurcell}

\begin{document}


\begin{abstract}
  We give uniform, explicit, and simple face-pairing descriptions of all the branched cyclic covers of the 3--sphere, branched over two-bridge knots. Our method is to use the bi-twisted face-pairing constructions of Cannon, Floyd, and Parry; these examples show that the bi-twist construction is often efficient and natural. Finally, we give applications to computations of fundamental groups and homology of these branched cyclic covers.
\end{abstract}

\maketitle

\section[Introduction]{Introduction}\label{sec-intro}

Branched cyclic covers of $\Sth$ have played a major role in topology, and continue to appear in a wide variety of contexts. For example, branched cyclic covers of $\Sth$ branched over two-bridge knots have recently appeared in combinatorial work bounding the Matveev complexity of a 3-manifold \cite{petronio-vesnin}, in algebraic and topological work determining relations between $L$-spaces, left-orderability, and taut foliations \cite{gordon-lidman, boyer-gordon-watson, yhu}, and in geometric work giving information on maps of character varieties \cite{nagasato-yamaguchi}. They provide a wealth of examples, and a useful collection of manifolds on which to study conjectures. Given their wide applicability, and their continued relevance, it is useful to have many explicit descriptions of these manifolds. 

In this paper, we give a new and particularly elegant construction of the branched cyclic covers of two-bridge knots, using the bi-twist construction of \cite{CFP4}. While other presentations of these manifolds are known (see, for example \cite{Min82, Mul01}), we feel our descriptions have several advantages, as follows. 

First, they follow from a recipe involving exactly the parameters necessary to describe a two-bridge knot, namely, continued fraction parameters. Our descriptions apply uniformly to all two-bridge knots, and all branched cyclic covers of $\Sth$ branched over two-bridge knots. 

Second, they are obtained from a description of a two-bridge knot using a particularly straight-forward bi-twisted face pairing construction, as in \cite{CFP1,CFP2,CFP3,CFP4}. Bi-twisted face-pairings are known to produce all closed orientable 3-manifolds. The examples of this paper show, in addition, that bi-twist constructions are often efficient and natural. While a generic face-pairing will yield a pseudomanifold, which, with probability 1, will not be an actual manifold~\cite{Dun06}, bi-twisted face-pairings avoid this problem. 
In this paper, we will review necessary information on bi-twisted face pairings, so no prior specialized knowledge is required to understand our constructions. 

Third, our description leads to immediate consequences in geometric group theory. We obtain a simple proof of the fact that the fundamental group of the $n$-fold branched cyclic covering of $\Sth$, branched over a two-bridge knot, has a cyclic presentation. Our description also gives immediate presentations of two well known families of groups, the Fibonacci and Sieradski groups. These are known to arise as fundamental groups of branched cyclic covers of $\Sth$ branched over the figure-eight and trefoil knots, respectively. These groups have received considerable attention from geometric group theorists; see, for example~\cite{Cav98} for further references, and Section~\ref{sec-history} for more history.
Our methods recover the fact that the first homology groups of Sieradski manifolds are periodic. We also give a proof that their fundamental groups are distinct using Milnor's characterization of these spaces. We consider orders of abelianizations of Fibonacci groups, as well. These orders form an interesting sequence related to the Fibonacci sequence, which we shall see.

\subsection{Bi-twisted face-pairing description}\label{sec-facepairing}

We will see that the bi-twist description of any two-bridge knot is encoded as the image of the north-south axis in a ball labeled as in Figure~\ref{fig-tbf1}, along with an associated vector of integer multipliers. For the branched cyclic cover, the description is encoded by adding additional longitudinal arcs to the sphere.
We now describe the construction briefly, in order to state the main results of the paper. A more detailed description of the construction, with examples, is given in Section~\ref{sec-trefoil}.

\begin{figure}[h!]
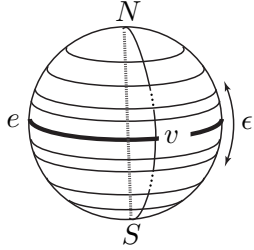
\caption{The model face-pairing: a faceted 3-ball with dotted central axis and reflection face-pairing $\zep\from\partial B^3\to \partial B^3.$}
\label{fig-tbf1}
\end{figure}

Begin with a finite graph $\Zc$ in the $2$-sphere $\St = \bd \Bth$ that is the union of the equator $e$, one longitude $NS$ from the north pole $N$ to the south pole $S$, and $2k \ge 0$ latitudinal circles, such that $\Zc$ is invariant under reflection $\zep\from\St \to \St$ in the equator. Then $\Zc$ divides $\St$ into $2(k+1)$ faces that are paired by $\zep$. This face-pairing is shown in Figure~\ref{fig-tbf1}.

As with any face-pairing, the edges fall into edge cycles. The equator $e$ forms one edge cycle $c_0$ since the reflection $\zep$ leaves $e$ invariant. Each other edge of the graph is matched with its reflection to form another edge cycle $c_i$. We number these edge cycles from $0$ through $2k+1$, with even numbers associated with latitudinal edges, as indicated in Figure~\ref{fig-tbf2}.

\begin{figure}[h!]
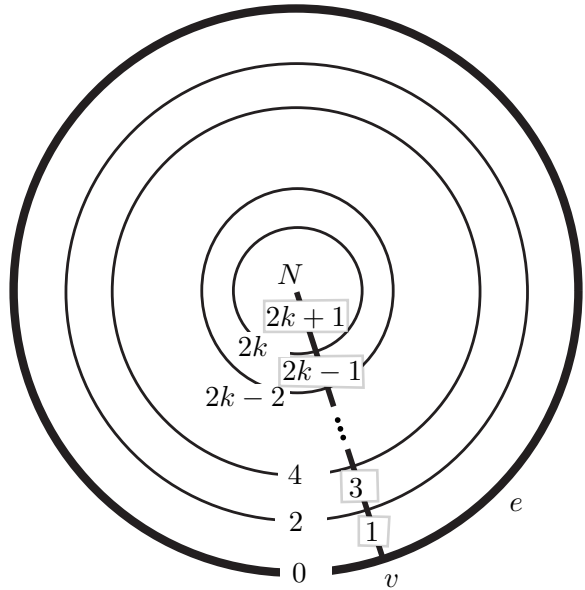
\caption{The northern hemisphere, with edge cycles numbered.}
\label{fig-tbf2}
\end{figure}

Now choose nonzero integer \emph{multipliers}, denoted $m_0$, $m_1$, $\ldots$, $m_{2k}$, $m_{2k+1}$, for the edge cycles $c_i$. In the case at hand, restrict the choice of multipliers $m_i$ as follows. Each latitudinal edge cycle $c_{2i}$ is assigned either $+1$ or $-1$ as multiplier. Each longitudinal edge cycle $c_{2i+1}$ may be assigned any integer multiplier $m_{2i+1}$ whatsoever, including $0$. The multiplier $m_{2i+1} =0$ is usually forbidden, but in this case indicates that the two edges of edge-class $c_{2i+1}$ must be collapsed to a point before the bi-twist construction is engaged.

Finally, for the general bi-twist construction, we obtain a closed manifold $M(\zep, m)$ by taking the following quotient. First, subdivide each edge in the edge cycle $c_i$ into $|c_i|\cdot|m_i|$ sub-edges. Insert an additional edge between each adjacent positive and negative edge, if any. Then twist each sub-edge by one sub-edge in a direction indicated by the sign of $m_i$. Finally, apply the face-pairing map $\zep$ to glue bi-twisted faces. This is the bi-twist construction.

In Theorem~\ref{thm-qisS3}, we prove that the bi-twist manifold $M(\zep,m)$ described above is the $3$-sphere $\Sth$. The image of the north--south axis in $\Sth$ is a two-bridge knot.
In fact, we prove more. Recall that every two-bridge knot is the closure of a rational tangle. See~\cite{Kau02} for an elementary exposition. A rational tangle is determined up to isotopy by a single rational number, which we call the rational number invariant of the tangle. There are two natural ways to close a tangle so that it becomes a knot or link, the numerator closure and the denominator closure. The full statement of Theorem~\ref{thm-qisS3} is below.

\begin{named}{Theorem~\ref{thm-qisS3}}
The bi-twist manifold $M(\zep,m)$ is the $3$-sphere $\Sth$. The image of the north--south axis in $\Sth$ is the two-bridge knot which is the numerator closure of the tangle $T(a/b)$ whose rational number invariant $a/b$ is
\[
2\cdot m_0+\cfrac{1}{2\cdot m_1+\cfrac{1}{\cdots+\cfrac{1}{2\cdot
m_{2k}+\cfrac{1}{2\cdot m_{2k+1}}}}}.
\]
\end{named}

\begin{remark}
The $2$'s in the continued fraction indicate that the tangle is constructed using only full twists instead of the possible mixture of full and half twists. 
\end{remark}

\begin{example}
The simplest case, with only equator and longitude, yields the trefoil and figure-eight knots, as we shall see in Theorem~\ref{thm:trefoil}. Simple subdivisions yield their branched cyclic covers, the Sieradski \cite{Sie86} and Fibonacci \cite{Ves96b} manifolds.
\end{example}

\begin{definition}\label{def-normlized}
We say that the multiplier function $m$ is \emph{normalized} if the following hold:
\begin{enumerate}
\item $m_{2k+1}\ne 0$, and
\item if $m_{2i+1}=0$ for some $i\in\{0,\ldots,k-1\}$, then
$m_{2i}=m_{2i+2}$.
\end{enumerate}
\end{definition}

With this definition, the previous theorem and well-known results involving two-bridge knots yield the following corollary.

\begin{named}{Corollary~\ref{cor-denomclosure}}
Every normalized multiplier function yields a nontrivial two-bridge knot.  Conversely, every nontrivial two-bridge knot $K$ is realized by either one or two normalized multiplier functions.  If $K$ is the numerator closure of the tangle $T(a/b)$, then it has exactly one such realization if and only if $b^2\equiv 1 \text{ mod } a.$
\end{named}

Notice that the $n$-th branched cyclic covering of $\Sth$, branched over $K$, can be obtained by unwinding the description $n$ times about the unknotted axis that represents $K$, unwinding the initial face-pairing as in Figure~\ref{fig-tbf22}. This has the following application.

\begin{named}{Theorem~\ref{thm-Gn}}
The fundamental group of the $n$-th branched cyclic covering of $\Sth$, branched over $K$, has a cyclic presentation.
\end{named}

\begin{problem}\label{prob:knots}
How should one carry out the analogous construction for arbitrary knots?
\end{problem}

\subsection{The Fibonacci and Sieradski manifolds}\label{sec-Fibonacci}

Since the knots in the face-pairing description
appear as the unknotted axis in $\Bth$, it is easy to unwind $\Bth$ around the axis to obtain face pairings for the branched cyclic coverings of $\Sth$, branched over the trefoil knot and the figure-eight knot. For the trefoil knot, the $n$-th branched cyclic cover $S_n$ is called the $n$-th Sieradski manifold. For the figure-eight knot, the $n$-th branched cyclic cover $F_n$ is called the $n$-th Fibonacci manifold.

We will prove:

\begin{named}{Theorem~\ref{thm:FnSnPresentation}}
The fundamental group $\pi_1(F_n)$ is the $2n$-th Fibonacci group with presentation
\[
\langle x_1,\ldots,x_{2n}\mid x_1x_2=x_3,\,x_2x_3=x_4,\,\ldots,\,x_{2n-1}x_{2n}=x_1,\,x_{2n}x_1=x_2\rangle.
\]
The fundamental group $\pi_1(S_n)$ is the $n$th Sieradski group with presentation
\[
\langle y_1,\ldots,y_n\mid y_1 = y_2y_n,\,y_2 =
y_3y_1,\,y_3=y_4y_2,\,\ldots,\,y_n = y_1y_{n-1}\rangle.
\]
\end{named}

\begin{remark}
The group presentations are well-known once the manifolds are recognized as branched cyclic covers of $\Sth$, branched over the figure-eight knot and the trefoil knot. But these group presentations also follow immediately from the description of the bi-twist face-pairings, as we shall see.
\end{remark}

The first homology of the Sieradski manifolds has an intriguing periodicity property, which is well-known (see 
for example Rolfsen~\cite{Rol76}). In particular, it is periodic of period $6$.
The following theorem, concerning their fundamental groups, is not as well known; it is difficult to find in the literature. We give a proof using Milnor's characterization of these spaces. 

\begin{named}{Theorem~\ref{thm:siergps}}
No two of the Sieradski groups are isomorphic.  Hence no two of the branched cyclic covers of $\Sth$, branched over the trefoil knot, are homeomorphic.
\end{named}

\subsection{Organization}
This paper is organized as follows. In Section~\ref{sec-trefoil}, we will give a more careful description of the bi-twisted face-pairing, and work through the description for two examples, which will correspond to the trefoil and figure-eight knots. 

In Section~\ref{sec-Heegaard}, we recall many of the results in our previous papers \cite{CFP3, CFP4} to make explicit the connections between face-pairings, Heegaard splittings, and surgery descriptions of 3-manifolds. We apply these to the examples of bi-twisted face pairings given here, to give surgery descriptions. We use these descriptions in Section~\ref{sec-twobridge} to prove that our constructions yield two-bridge knots. The proofs of the main geometric theorems are given in this section.

In Section~\ref{sec-cyclicpresentations} we turn to geometric group theory. We prove that our presentations easily lead to well-known results on presentations of fundamental groups. We also give results on Fibonacci and Sieradski groups in this section.

Finally, Section~\ref{sec-history} explains some of the history of these problems. 

\subsection{Acknowledgements}
 We thank the referees of an earlier version of this paper for numerous helpful comments. Purcell is partially supported by NSF grant DMS-1252687.

\subsection{Dedication}
Though LeeR Lambert spent his life as an actuary and a musician and was a loving father of nine girls and one boy, LeeR had always wanted to earn an advanced degree as a mathematician. With the encouragement of his wife, he earned his Ph.D.\ in mathematics at the age of 68. Many of the results of this paper appeared in his BYU Ph.D.\ dissertation. At the age of 71, LeeR died of bone cancer. We miss you, LeeR.

\section{Bi-twisted face-pairing: trefoil and figure-eight knots}\label{sec-trefoil}

In this section we step through the bi-twisted face-pairing description more carefully. We believe it will be most useful to work through a pair of examples. We will see in subsequent sections that these examples lead to Fibonacci and Sieradski manifolds.

As an example, consider the simplest model, shown in Figure~\ref{fig-tbf3}(a). The graph has three edges and three vertices, and divides the sphere into two singular ``triangles'', which are then matched by reflection $\zep$ in the equator $e$.

\begin{figure}
 \begin{tabular}{ccc}
\def\svgwidth{1.6in}
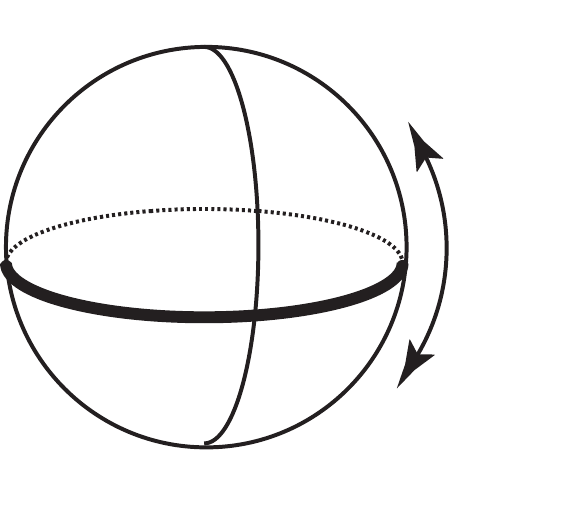 &
\def\svgwidth{1.2in}
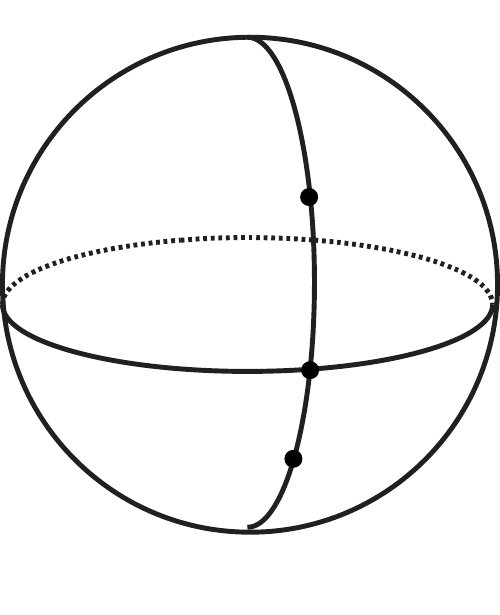 \hspace{.25in} &
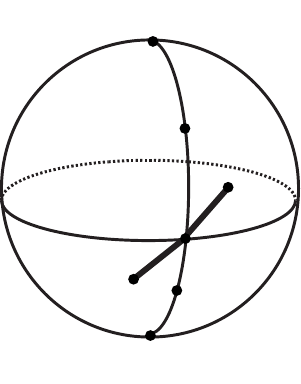 \\
(a) & (b) & (c)
  \end{tabular}
\caption{(a) A faceted 3-ball with vertices $N$, $v$, and $S$ and edges $Nv$, $Sv$, and $e$. (b) Subdivisions for $M_+$. (c) Subdivisions for $M_-$.}
\label{fig-tbf3}
\end{figure}

Bi-twisted face-pairings require 
an integer multiplier for each edge cycle. For this simple model there are two edge cycles, namely the singleton $c_0 = \{e\}$ 
and the pair $c_1 = \{Nv,\,Sv\}$.
We will see that multiplying every multiplier by $-1$ takes the knot which we construct to its mirror image.  So up to taking mirror images, the two simplest choices for multipliers are $m_0 = \pm1$ for $c_0$ and $m_1 = 1$ for $c_1$. The bi-twist construction requires that each edge in the cycle $c_i$ be subdivided into $|c_i|\cdot |m_i|$ subedges. When both positive and negative multipliers appear on edges of the same face, we must insert an additional edge, called a \emph{sticker}, between a negative and positive edge in a given, fixed orientation of $\St$. We will use the clockwise orientation.

With the facets modified as described in the previous paragraph, we are prepared for the bi-twisting.
Twist each subedge of each face by one subedge before applying the model map $\zep$. Edges with positive multiplier are twisted in the direction of the fixed orientation. Edges with negative multiplier are twisted in the opposite direction. The stickers resolve the twisting conflict between negative and positive subedges. A sticker in the domain of the map splits into two subedges. A sticker in the range of the map absorbs the folding together of two subedges.

We denote by $M_+$ the face pairing in which both multipliers are $+1$ and by $M_-$ the face pairing where one multiplier is $+1$ and the other is $-1$. The two results are shown in Figure~\ref{fig-tbf3}(b) and (c).

After this subdivision, the faces can be considered to have $5$ edges for $M_+$ and $7$ edges for $M_-$. Before making the identification of the northern face with the southern face, we rotate the $5$-gon one notch (= one edge = one fifth of a turn, combinatorially) in the direction of the given orientation on $\St$ before identification. We rotate the edges of the $7$-gon with positive multiplier one notch (= one edge = one seventh of a turn, combinatorially) in the direction of the orientation before identification.  The edges with negative multiplier are twisted one notch in the opposite direction.  The stickers absorb the conflict at the joint between positive and negative.
Thus the face-pairings $\zep_+$ and $\zep_-$ in terms of the edges forming the boundaries of the faces are given as follows.

For $M_+$:
\[
\zep_+:\begin{pmatrix} av&e &va&aN&Na\\
                                e&va'&a'S&Sa'&a'v
    \end{pmatrix}
\]
      
For $M_-$:
\[ \zep_-:\begin{pmatrix} av&e   &vx&xv &va &aN&Na\\
  vx'&x'v&e  &va'&a'S&Sa'&a'v
    \end{pmatrix}
\]

The bi-twist theorem \cite[Theorem 3.1]{CFP4} implies that the resulting  identification spaces are closed manifolds, which we denote by $F_1$ for $M_+$ and $S_1$ for $M_-$. We shall see that both of these manifolds are $\Sth$, and thus topologically uninteresting. But as face-pairings, these identifications are wonderfully interesting because the north-south axis from $\Bth$ becomes the figure-eight knot $K_+$ in $F_1$ and becomes the trefoil knot $K_-$ in $S_1$. We prove this in Theorem~\ref{thm:trefoil}.

\section[Heegaard splittings]{Pseudo-Heegaard splittings and
surgery diagrams}\label{sec-Heegaard}

In order to recognize the quotients of $\Bth$ described in Section~\ref{sec-intro} as the $3$-sphere and to recognize the images of the
north-south axis as two-bridge knots,  we need to make more explicit the connections between face-pairings, Heegaard splittings, and surgery descriptions of $3$-manifolds, described in our previous papers~\cite{CFP3,CFP4}. We use this description to transfer knots from the face-pairing description to the surgery descriptions.

\subsection{The pseudo-Heegaard splitting}
We begin with the following information:

\begin{description}\setlength{\itemsep}{.05in}
\item[$B$] a faceted $3$-ball which we identify with $\Bth = [0,1]\cdot \St$ ($\cdot$ = scalar multiplication).
\item[$\Zc\subset \bd B = \St$] the $1$-skeleton of $B$, a connected, finite graph with at least one edge.
\item[$\Zd$] the dual $1$-skeleton, consisting of a cone from the center $0$ of $B$ to points of $\bd B$, one in the interior of each face of $B$.
\item[$N$] a regular neighborhood of $\Zc$ in $\bd B$.
\item[{$N_\Zc = [3/4, 1]\cdot N$}] a regular neighborhood of $\Zc$ in $B$. 
\item[$N_\Zd = \hbox{cl}(B - N_\Zc)$] a regular neighborhood of $\Zd$ in $B$.
\end{description}

Add extra structure to $N$ and $N_\Zc$ as follows.

First, from each vertex $v$ of $\Zc$, we extend arcs from $v$ to $\bd N$, one to each local side of $\Zc$ at $v$ so that the interiors of these arcs are mutually disjoint. Label these arcs \emph{red}. Figure~\ref{fig-tbf6} shows this for the simplest model described above.

\begin{figure}[ht]
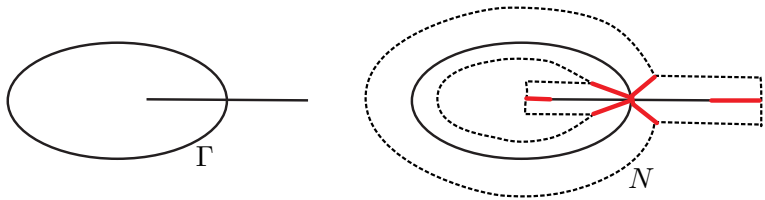
\caption{The addition of new red arcs.}
\label{fig-tbf6}
\end{figure}

Next, momentarily disregarding both the vertices and edges of $\Zc$, we view the red arcs as subdividing $N$ into quadrilaterals (occasionally singular at the arc ends), every quadrilateral having two sides in $\partial N$ and two sides each of which is the union of two (or one in the singular case) of these red arcs, as on the left of Figure~\ref{fig-tbf7}.

Every such quadrilateral contains exactly one edge of $\Zc$.  We cut these quadrilaterals into half-quadrilaterals by arcs transverse to the corresponding edge of $\Zc$ at the middle of that edge. Label these transverse arcs \emph{blue}. For the simplest model, this is shown in Figure~\ref{fig-tbf7}, right.

\begin{figure}[ht]
\centerline{\includegraphics{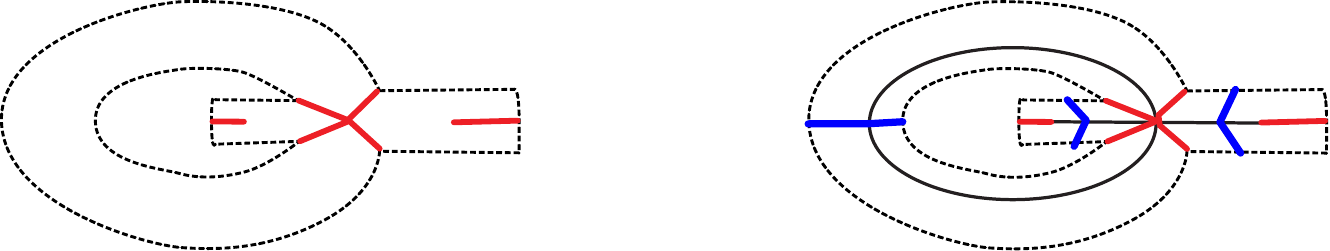}}
\caption{The addition of blue transverse arcs.}
\label{fig-tbf7}
\end{figure}

If we cut $N$ along the new red arcs and blue transverse arcs, multiply by the scalar interval $[3/4,1]$, and desingularize, we obtain cubes, each containing exactly one vertex of $\Zc$ in its boundary. Endow these cubes with a cone structure, coned to its vertex in $\Zc$. See Figure~\ref{fig-tbf8}.

\begin{figure}[ht]
  \centerline{\includegraphics{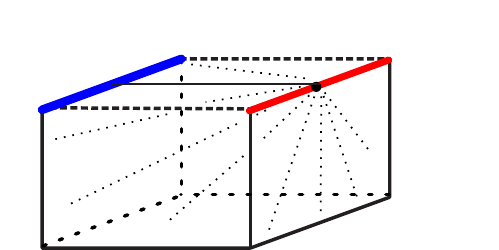}}
\caption{The cone structure.}
\label{fig-tbf8}
\end{figure}

Finally, we assume that $\zep\from\bd B \to \bd B$ is an orientation-reversing face-pairing, based on the faceted $3$-ball $B$, that respects all of this structure as much as possible: faces are paired, $N$ is invariant under the pairing, the regions bounded by the new arcs, the transverse arcs, the boundary of $N$, and $\Zc$ are paired by $\zep$, and cone structures are preserved.

\begin{definition}\label{def:handle-disks}
Let $C_\Zc$ be the union of the products of the transverse arcs with $[3/4,1]$. Let $C_\Zd = N_\Zd \cap (\bd B)$. Define $D_\Zc = C_\Zc/\zep$, $D_\Zd = C_\Zd/\zep$, $H_\Zc = N_\Zc/\zep$, and $H_\Zd = N_\Zd/\zep$, and let $\zd = \bd D_\Delta$ and $\zc = \bd D_\Gamma$.
\end{definition}

The following is essentially contained in \cite[Theorem~4.2.1]{CFP3}.

\begin{theorem}
The space $H_\Zd$ is a handlebody with one handle for each face pair of $B$. The set $D_\Zd$ is a disjoint union of disks that form a complete set of handle disks for $H_\Zd$; the curves $\zd$ form a complete set of handle curves.

The space $H_\Zc$ is a handlebody if and only if $M(\zep) = B/\zep$ is a $3$-manifold. In that case, $D_\Zc$ is a disjoint union of disks that form a complete set of handle disks for $H_\Zc$ and $\zc$ form a complete set of handle curves. Whether $M(\zep)$ is a manifold or not, the disks of $D_\Zc$
cut $H_\Zc$ into pieces $X_i$, each containing exactly one vertex $v_i$ of $M(\zep)$, and each $X_i$ is a cone $v_iS_i$, where $S_i$ is a
closed orientable surface. The space $M(\zep)$ is a manifold if and
only if each $S_i$ is a $2$-sphere. (The cone structure on $X_i$ uses
the cone structures of the pieces described above.)
\end{theorem}

\begin{terminology}
Even when $M(\zep) = B/\zep$ is not a manifold, we call the disks of $D_\Zc$ handle disks for $H_\Zc$ and the curves $\zc = \bd D_\Zc$ handle curves for $H_\Zc$. We call $H_\Zc$ a pseudo-handlebody and the pair $(H_\Zc,H_\Zd)$ a pseudo-Heegaard splitting for $M(\zep)$.
\end{terminology}

All bi-twist manifolds based on the face-pairing $(B,\zep)$ have Heegaard splittings and surgery descriptions that can be based on any unknotted embedding of $H_\Zd = N_\Zd/\zep$ in $\Sth = \Rth \cup \{\infty\}$. The closure of the complement is then also a handlebody, which we shall denote by $H$. We describe here a particular unknotted embedding of $\Hd$ in $\Sth$, and illustrate with the constructions from Section~\ref{sec-facepairing}, especially those of Section~\ref{sec-trefoil}.

Note that $N_\Zd = ([0,3/4]\cdot\St)\cup([3/4,1]\cdot C_\Zd)$, where $[0,3/4]\cdot \St$ is, of course, a $3$-ball, and $[3/4,1]\cdot C_\Zd$ is a family of chimneys attached to that $3$-ball, as in Figure~\ref{fig-tbf9}(a).

\begin{figure}[ht]
  \begin{center}
\begin{tabular}{cc}
  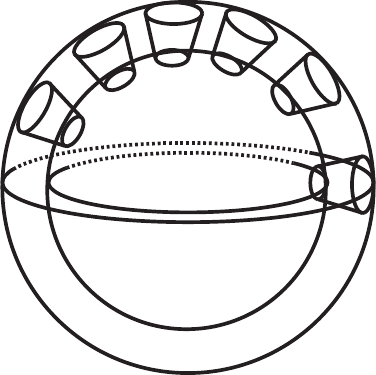 &
  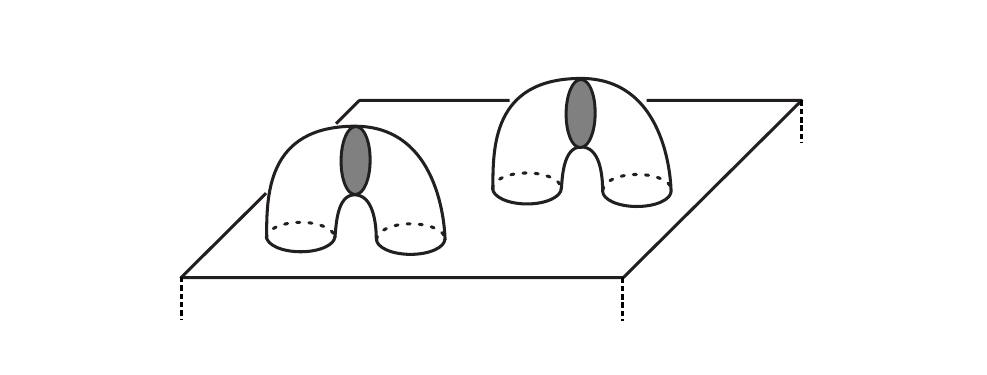 \\
  (a) & (b)
\end{tabular}
  \end{center}
\caption{(a) The ball with chimneys $N_\Zd$. (b) The handlebody $H_\Zd$.}
\label{fig-tbf9}
\end{figure}

The space $H_\Zd$ is formed by identifying the tops of those chimneys in pairs. We may therefore assume $H_\Zd$ is embedded in $\Sth = \Rth \cup \{\infty\}$ as shown in Figure~\ref{fig-tbf9}(b).  We identify $[0,3/4)\cdot \St$ with $\Rt\times (-\infty,0)\subset \Rth$.  The 2-sphere $(3/4)\cdot \St$ minus one point is identified with $\Rt\times \{0\}\subset \Rth$.  The chimneys with tops identified become handles.

\subsection{Pseudo-Heegaard splittings of our examples}
For the constructions of Section~\ref{sec-facepairing} and~\ref{sec-trefoil}, we now determine the curves $\zd$ and $\zc$ on the handlebody $H_\Zd$.

Begin with the simple face-pairing description of Section~\ref{sec-trefoil}. The handlebody $H_\Zd$ is embedded in $\RR^3\cup\{\infty\}$ as above, with the plane $\RR^2\times\{0\}$ identified with $(3/4)\cdot\Sth$ minus a point. Sketch the graph $(3/4)\cdot \Gamma$ on $\RR^2\times\{0\}$, with the vertex $(3/4)\cdot v$ at $\infty$, as in Figure~\ref{fig:HTrefoil}, left. There is just one pair of faces, hence just one handle in this case, as shown. Thus $D_\Zd$ is a single disk with boundary $\zd$, shown in the figure, left.

\begin{figure}
  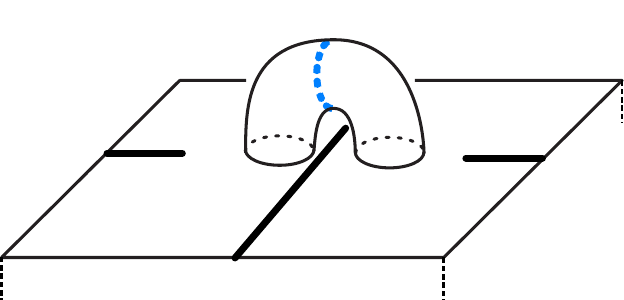
  \hspace{.3in}
  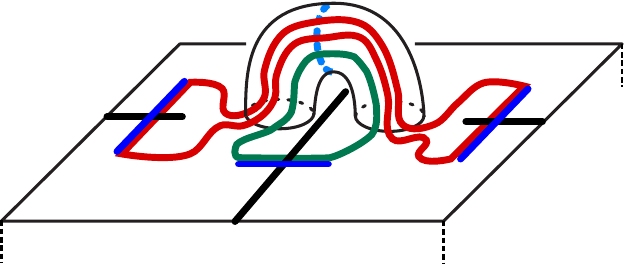
  \caption{Left: The graph $(3/4)\cdot\Gamma$ and curve $\delta$ for the simple example. Right: Curves $\gamma$ added in, running partly along blue transverse arcs.}
  \label{fig:HTrefoil}
\end{figure}

We need to determine the curves $\zc = \bd D_\Zc$. Recall that $D_\Zc = C_\Zc/\zep$, and the disks $C_\Zc$ consist of the union of the products of the blue transverse arcs with $[3/4, 1]$. Thus curves in $\zc$ will contain blue transverse arcs, as well as arcs along the handles of $H_\Zd$, running from the blue transverse arcs to a curve $\delta_j$.

In the case of the simple example, following the action of $\zep$, we see that the transverse arc $\tau_0$ dual to the edge $e$ gives a single simple closed curve $\gamma_0$ that follows $\tau_0$, then connects the endpoints of $\tau_0$ via an arc that runds over the single handle of $H_\Zd$. The two transverse arcs dual to $Nv$ and $Sv$ are identified by $\zep$. Thus endpoints of these arcs are connected by arcs running over the handle. We obtain a simple closed curve $\gamma_1$. This is shown in Figure~\ref{fig:HTrefoil}, right.

The general picture, for the construction of Section~\ref{sec-facepairing}, follows similarly. We summarize in a lemma.

\begin{lemma}\label{lemma:HandleCurves}
Let $\Zc$ and $\zep$ be as in subsection~\ref{sec-facepairing}, with $\Zc$ the union of the equator $e$, one longitude $NS$ from the north pole $N$ to the south pole $S$, and $2k\geq 0$ latitudinal circles, such that $\Zc$ is invariant under reflection $\zep$ in the equator. Then the handle curves on $H_\Zd$ for this face pairing are as follows.
\begin{enumerate}
\item There are $k+1$ handles of $H_\Zd$, corresponding to the $k+1$ regions in the complement of $\Gamma$ in the northern hemisphere, each running from the region to its mirror region in the southern hemisphere. These give curves $\delta_0, \dots, \delta_k$ encircling the handles.
\item The transverse arc dual to the edge $e$ gives a curve $\gamma_0$ with endpoints connecting to itself over the handle corresponding to the faces on either side of $e$, which are identified by $\zep$.
\item Each latitudinal arc distinct from $e$, if any, is joined to its mirror over two handles, one for each face on opposite sides of the latitudinal edge. These give curves $\gamma_{2i}$, $i=1, \dots, k$, with index corresponding to the edge label as in Figure~\ref{fig-tbf2}.
\item Each transverse arc dual to a longitudinal arc is joined to its mirror over a handle corresponding to the region on either side of that arc. These give curves $\gamma_{2i+1}$, $i=0, \dots, k$, again with index corresponding to edge label as in Figure~\ref{fig-tbf2}.
\end{enumerate}
\end{lemma}

Curves parallel to those of Lemma~\ref{lemma:HandleCurves} are illustrated in Figure~\ref{fig-tbf16}. Note these curves have been pushed slightly to be disjoint, in a manner described in Subsection~\ref{subsec:surgery}. 

\begin{figure}
  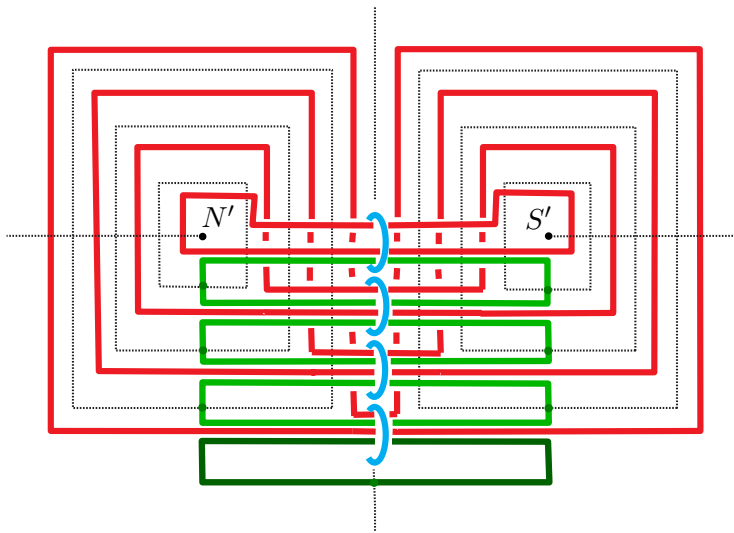
\caption{The curves $\delta$ and $\gamma$ in $H_\Zd$. Curves $\delta$ are shown in blue, $\gamma_0$ in dark green at the bottom of the diagram, curves of $\gamma$ correpsonding to to latitudinal transverse arcs are in green, and curves of $\gamma$ corresponding to longitudinal transverse arcs are in red.}
\label{fig-tbf16}
\end{figure}

\subsection{The surgery description}\label{subsec:surgery}
We assume now that we are given a bi-twist construction based on
$(B,\zep)$. We are given the following information:

\begin{description}\setlength{\itemsep}{.05in}
\item[$c_1,\ldots,c_k$] the edge cycles of $\zep$.
\item[$m = \{m_1,\ldots,m_k\}$] a set of nonzero integer multipliers assigned to these edge cycles.
\item[$\zep_m:\bd B \to \bd B$] the associated bi-twist face-pairing. 
\item[$M(\zep,m)=M(\zep_m) = B/\zep_m$] the resulting bi-twist manifold.
\end{description}

The set $\zd = \bd D_\Zd$ is a disjoint union of simple closed handle curves $\zd_1,\ldots,\zd_g$ for $ H_\Zd$, one for each face pair of $\zep$. We first push each $\zd_i$ slightly into $\Rth\setminus H_\Zd$ to a curve $\zd_i'$. We let $V_i$ denote a solid torus neighborhood of $\zd_i'$ in $\Rth\setminus H_\Zd$, remove it, and sew a new solid torus $V_i'$ back in with meridian and longitude reversed ($0$-surgery on each $\zd_i'$). The curve $\zd_i$ now bounds a disk $E_i$, disjoint from $H_\Zd$, consisting of an annulus from $\zd_i$ to $\bd V_i'$ and a meridional disk in $V_i'$. The result is a new handlebody
\[
H' = [\hbox{cl}(\Sth\setminus H_\Zd)\setminus \cup V_i] \cup [\cup V_i']
\]
with the same handle curves $\zd_1,\ldots,\zd_g$ as $H_\Zd$ and with
handle disks $E_1,\ldots,E_g$. The union $H_\Zd \cup H'$ is
homeomorphic to $(\St \times \So)\# \cdots \# (\St \times \So)$.

The set $\zc = \bd D_\Zc$ is a disjoint union of simple closed curves $\zc_1,\ldots, \zc_k$ on $\bd H'$, one for each edge class of $\zep$. We push each $\zc_j$ slightly into $\hbox{int}(H')\setminus (\cup V_i')$ to a curve $\zc_j'$. On each $\zc_j'$ we perform $\hbox{lk}(\zc_j,\zc_j') + (1/m_j)$ surgery. Note from Lemma~\ref{lemma:HandleCurves} that in our applications, the
curves  $\zc_j$ will be unknotted and the curves $\zc'_j$ will have linking number $0$ with them. 

These surgeries modify $H'$ to form a new handlebody $H''$. By \cite[Theorem~4.3]{CFP4}, 
$(H_\Zd,H'')$ is a Heegaard splitting for $M(\zep,m)$ (or, because of ambiguities associated with
orientations, the manifold $M(\zep,-m)$, with $-m =\{-m_1,\ldots,-m_k\}$, which is homeomorphic with $M(\zep,m)$).

For our purposes, it is important to see that these surgeries can be realized by an explicit homeomorphism from $H'$ to $H''$ defined by Dehn-Lickorish moves. To that end, we enclose $\zc_j'$ in a solid torus neighborhood $U_j$ that is joined to $\zc_j$ by an annulus $A_j$. We remove $U_j$ and cut the remaining set along $A_j$. Let $A_j'$ denote one side of the cut. We may parametrize a neighborhood of $A_j'$ by $(\zth, s, t)$ where $\zth\in \R ({\text{mod } }2\pi)$ is the angle around the circle $\zc_j$, $s\in [0,1]$ is the depth into $H'$, and $t\in [0,1]$ is distance from $A_j'$. Then one twists this neighborhood of $A_j'$ by the map $(\zth,s,t)\mapsto (\zth + (1-t)\cdot m_j\cdot 2\pi,s,t)$ before reattaching $A_j'$ to its partner $A_j''$ to reconstitute $A_j$. This twisting operation defines a homeomorphism $\phi\from [H'\setminus(\cup U_j)]\to [H'\setminus(\cup U_j)]$. One then reattaches the solid tori $U_j$ via the homeomorphisms $\phi|_{\bd U_j}$ to form $H''$, with an extended homeomorphism $\Phi:H'\to H''$. The homeomorphism $\Phi$ is the identity except in a small neighborhood of $\zc$. The new handle disks are $\Phi(E_1),\dots,\Phi(E_g)$.

We apply this to obtain a surgery description for our construction. Recall from subsection~\ref{sec-facepairing} that our multipliers were chosen to be $\pm 1$ on latitudinal edge cycles, and any integer $m_i$ on longitudinal edges. We record the result in the following lemma.

\begin{lemma}\label{lemma:genSurgeryDescription}
Let $\Zc$ and $\zep$ be as in subsection~\ref{sec-facepairing}, with handle curves as in Lemma~\ref{lemma:HandleCurves}. Then the manifold $M(\zep,m)$ has the following surgery description.
\begin{enumerate}
\item There are $k+1$ simple closed curves $\delta_0', \dots, \delta_k'$, with each $\delta_j'$ parallel to $\delta_j$, pushed to the exterior of the handle of $H_\Zd$. Each $\delta_j'$ has surgery coefficient $0$.
\item Each curve of $\gamma$ corresponding to a latitudinal edge class $\gamma_{2i}$ appears with surgery coefficient $m_{2i} = \pm 1$, $i=0, \dots, k$.
\item Each curve of $\gamma$ corresponding to longitudinal edge class $\gamma_{2i+1}$ has surgery coefficient $1/m_{2i+1}$. If one of these multipliers is $0$, so that the edge collapses to a point and disappears as an edge class, we retain the corresponding curve, but with surgery coefficient $1/0=\infty$.
\end{enumerate}
\end{lemma}\qed

The curves are shown in Figure~\ref{fig-tbf16}. 

\subsection[The North-South Axis]{The knot as the image of the
north-south axis}\label{sec-axis}

It is now an easy matter to identify the image of the north-south axis in our bi-twist constructions. In particular, we want to recognize this curve in the associated surgery description of the manifold. The portion of the curve in the handlebody $H_\Zd$ is obvious. That portion in the handlebody $H_\Zc$ is simple, yet not so obvious. We need a criterion that allows us to recognize it.

To that end, suppose that $\Hc$ is a pseudo-handlebody with one vertex $x$. Recall that $\Hc \setminus \Dc$ has a natural cone structure from $x$. We say that an arc $\alpha$ in $\Hc \setminus \Dc$ is \emph{boundary parallel} if there is a disk $D$ in $\Hc \setminus \Dc$ such that $(\bd D) \cap (\hbox{int}(\Hc)) = \hbox{int}(\alpha)$ and $(\bd D)\cap (\bd \Hc)$ is an arc $\alpha'$.

\begin{lemma}\label{lemma:bdyParallel}
Suppose $a,b\in (\bd\Hc)\setminus \Dc$ with $a\ne b$. Then the arc $\alpha = ax\cup bx$ (using the cone structure) is boundary parallel, and any arc $\beta$ that has $a$ and $b$ as endpoints and is boundary parallel is, in fact, isotopic to $\alpha$.
\end{lemma}

\begin{proof}
The set $\Dc$ is a disjoint union of handle disks for $\Hc$, hence does not separate $\bd\Hc$. There is therefore an arc $\alpha'$ from $a$ to $b$ in $(\bd\Hc)\setminus \Dc$. The disk $x\alpha'$, which uses the cone structure, proves that $\alpha$ is boundary parallel. If $\beta$ is boundary parallel, as certified by disk $E$ and  arc $\beta'$, then we may first assume $\text{int}(E) \subset \text{int}(\Hc\setminus \Dc)$, then we may straighten $E$ so that, near $(\bd \Hc)\setminus \Dc$, $E$ is part of the cone over $\beta'$. The arc $\beta$ may be slid along $E$ near to $\beta'$, then isotoped along the cone over $\beta'$ until it coincides with $\alpha$.
\end{proof}

For our construction, we are primarily interested in a curve of the form $(Ov\cup Ow)/\zep_m$, where $O$ is the center of $B$ and $v$ and $w$ are vertices of $\Zc$, all of which are identified by $\zep_m$ to a single vertex $x$ in $H_\Gamma$. The set $(Ov \cup Ow) \cap \Hd$ is immediately apparent. However, we must identify $\zb = (v'v\cup w'w)/\zep_m$, where $v' = (3/4)\cdot v$ and $w' = (3/4)\cdot w$. The images of $v$ and $w$ in $\Hc$ are the single
vertex $x$ of $\Hc$, and the image of $\zb$ is a cone from $x$ in the cone structure on $\Hc \setminus \Dc$. Therefore, by Lemma~\ref{lemma:bdyParallel}, to identify $\zb$ it suffices to find a boundary parallel arc in $\Hc$ with endpoints $v'$ and $w'$.

The vertices $v' = (3/4)\cdot v$ and $w' = (3/4) \cdot w$ lie in $\Rt \times \{0\}$, disjoint from the disks $(3/4)\cdot \Dd$, i.e.\ the attaching disks of the handles of $\Dd$ in $\RR^2\times\{0\}$. Hence, there is an arc $\za'$ in $(\Rt\times\{0\})\setminus (3/4)\cdot \Dd$ from $v'$ to $w'$. Take the product of $\za'$ and a small closed interval with left endpoint $0$ in $\Rt\times [0,\infty)\subset \Rth$. We obtain a disk $D$ in the handlebody $H$ that is the closure of $\Sth \setminus \Hd$. This disk exhibits the complementary arc $\za \subset \bd D$ as boundary parallel in $H$. We fix this arc and construct the handlebodies $H'$ and $H''$. Provided that the annuli and tori used in constructing $H'$ from $H$ are chosen close enough to the curves $\zd =\bd \Dd$ to avoid $D$, the disk $D$ will also certify that $\za$ misses the handle disks $E_i$ of $H'$ so that $\za$ is boundary parallel in $H'$. If the annuli $A_j$ and tori $U_j$ are chosen close enough to $\zc = \bd \Dc$ to avoid $\za$ (but not $D$), then the homeomorphism $\Zf\from H'\to H''$ will fix $\za$, will take the disks $E_i$ to handle disks for $H''$, and the disk $\Zf(D)$ will show that $\za$ is boundary parallel in $H''$. Thus $(Ov' \cup Ow') \cup \za$
represents the curve $(Ov\cup Ow)/\zep_m$ as desired.

Now we add this axis to our surgery descriptions. For the simplest construction, with equator $e$ and longitudinal arc $NS$, and handle curves as shown in Figure~\ref{fig:HTrefoil}, the surgery description is obtained by pushing $\delta_0$ slightly into $H$. Let $N'=(3/4)\cdot N$ and $S'=(3/4)\cdot S$ on $(3/4)\cdot \Gamma \subset \RR^2\times\{0\}$. The arc $(ON' \cup OS')$ runs below the plane $\RR^2\times\{0\}$ in $H_\Zd$. To find the arc $\alpha$, we take an arc $\alpha'$ from $N'$ to $S'$ in $\RR^2\times\{0\}$ disjoint from the handle, and, fixing the endpoints, push this above $\RR^2\times\{0\}$ slightly. By the above discussion, this gives the desired arc of the axis $NS$. The surgery diagram and the axis are shown for this example in Figure~\ref{fig-tbf11}.

\begin{figure}
  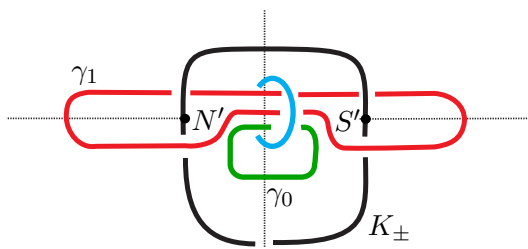
\caption{The surgery diagram for $K_{\pm}$, $S_1$, and $F_1$.}
\label{fig-tbf11}
\end{figure}

\section{Two-bridge knots}\label{sec-twobridge}

In this section, we prove that the image of the $NS$ axis in Figure~\ref{fig-tbf11} represents the figure eight knot in $\Sth$ when the surgery coefficient is taken to be $+1$, and  the trefoil knot in $\Sth$ when the coefficient is taken to be $-1$.

More generally, we prove that the $NS$ axis in the general construction represents a two-bridge knot in $\Sth$.

\subsection{Identifying the trefoil and figure-eight}\label{subsec:trefoil}

We will modify the surgery diagram of Figure~\ref{fig-tbf11} by means of Rolfsen twists. We remind the reader of the effect of a Rolfsen twist. We assume we are given an unknotted curve $J$ with surgery coefficient $p/q$ through which pass a number of curves, some of which are surgery curves  $K_i$ with surgery coefficients $r_i$, some of which may be of interest for some other reasons, such as our knot axis. We perform an $n$-twist on $J$. The curves passing through $J$ acquire $n$ full twists as a group.
The curve $J$ acquires the new surgery coefficient $p/(q+np)$; in particular, if $p = 1$, then a twist of $-q$ will change the coefficient to $\infty$, and any curve with a surgery coefficient $\infty$ can be removed from the diagram.
Finally, each surgery curve $K_i$ that passes through $J$ acquires the new surgery coefficient $r_i+n\cdot lk(J,K_i)^2$. 

\begin{theorem}\label{thm:trefoil}
The surgery description of $M(\zep, m)$ for the simple face pairing of Figure~\ref{fig-tbf3}(a) yields the manifold $\Sth$. The image of the north-south axis is the trefoil knot when $m=(-1, 1)$ and the figure-eight knot when $m=(1,1)$.
\end{theorem}

\begin{proof}
We apply Rolfsen twists to our surgery curves in the order $\zc_1$, $\zc_0$, and $\zd'$ to change their surgery coefficients, one after the other, to $\infty$. We trace the effect on the axis $K_{\pm 1}$, and show this in Figure~\ref{fig-tbf14}.

\begin{figure}
  \begin{center}
    \hspace{0.5in}
  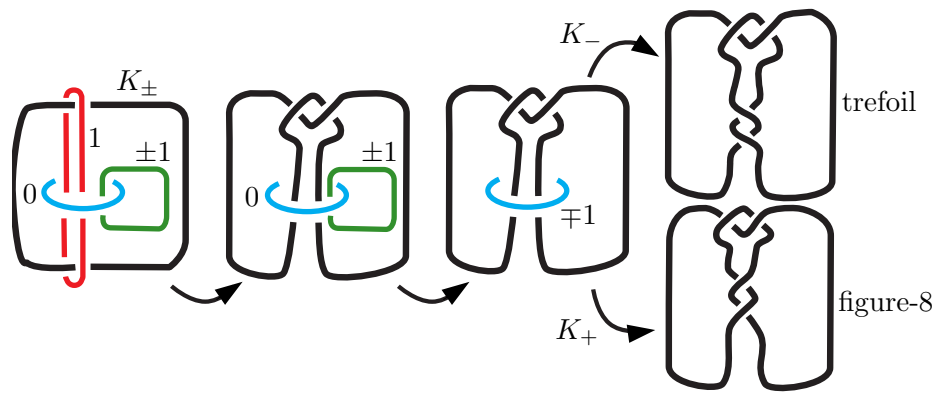
  \end{center}
\caption{Analyzing $S_1$, $F_1$, $K_-$, and $K_+$.}
\label{fig-tbf14}
\end{figure}

In detail, we first perform a $-1$ Rolfsen twist on $\zc_1$. This changes the surgery coefficient on $\zc_1$ to $\infty$ so that $\zc_1$ can be removed from the diagram. In the process, one negative full twist is added to the axis representing $K_\pm$.

We next perform a Rolfsen twist on $\zc_0$ to change its surgery coefficient to $\infty$ so that it too can be removed from the diagram. If the coefficient on $\zc_0$ was originally $1$, this twist must be a $-1$ twist. If the coefficient on $\zc_0$ was originally $-1$, this twist must be a $+1$ twist. The coefficient of this twist is added to the $0$ coefficient on the $\zd'$ curve. The axis is not affected.

Finally, we perform a Rolfsen twist on $\zd'$, opposite to its surgery coefficient $\mp 1$ so that its coefficient is changed to $\infty$. That makes it possible to remove $\zd'$ from the diagram. Since the diagram is now empty, we can conclude that the quotient manifold is $\Sth$.

This last twist adds a $\pm 1$ full twist to the axis and results in either the trefoil knot for the $(-1,1)$ multiplier pair or the figure-eight knot for the $(1,1)$ multiplier pair.
\end{proof}

\subsection{The general case}\label{sec-generalsurgery}

Having analyzed the simplest model face-pairing, we proceed to the general case. Thus we consider the $2$-sphere $\St = \bd\Bth$ subdivided by one longitude, the equator $e$, $k\ge 0$ latitudinal circles in the northern hemisphere, and their reflections in the southern hemisphere. As usual, we pair faces by reflection in the equator. There are $k+1$ face-pairs in this model face-pairing.

The general surgery description is given in Lemma~\ref{lemma:genSurgeryDescription}, and illustrated in Figure~\ref{fig-tbf16}. Subsection~\ref{sec-axis} tells us how to recognize the image of the north-south axis in this diagram. It is the union of a boundary parallel arc below the plane $\Rt \times \{0\}$ from $N'$ to $S'$ and a boundary parallel arc above the plane $\Rt \times \{0\}$ from $N'$ to $S'$. Straightening this axis curve and the surgery diagram, we obtain the diagram in Figure~\ref{fig-tbf17}.

\begin{figure}
  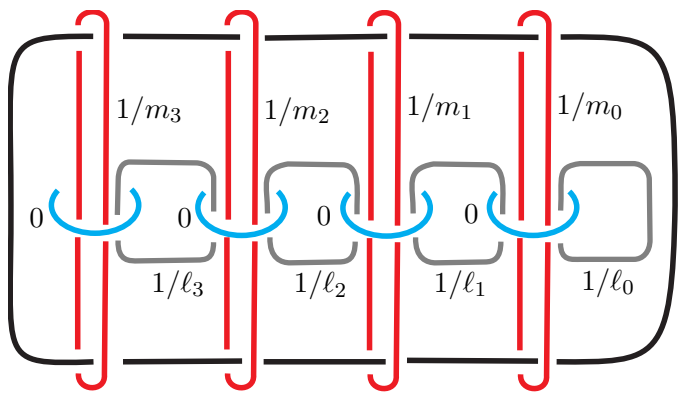
\caption{The surgery diagram.}
\label{fig-tbf17}
\end{figure}

Recall that the integers $m_{2i+1}$ are arbitrary --- positive, negative, or zero. The integers $m_{2i}$ are either $+1$ or $-1$. Note that the surgery curves fall naturally into three families, each with $k+1$ curves: the $\delta$ curves, circling the handles with surgery coefficients $0$, the latitudinal curves, linking the $0$-curves together in a chain and having coefficients $1/m_{2i} = \pm 1$, and the longitudinal curves with coefficients $1/m_{2i+1}$. Each of these curve families has a natural left-to-right order, as in the figure. To simplify notation, we denote the latitudinal curves from left to right by $L_k$, $L_{k-1}$, $\ldots$, $L_1$, $L_0$, and let the corresponding surgery coefficients be denoted $1/\ell_k,$, $\ell_{k-1}$, $\ldots$, $1/\ell_i$, $1\ell_0$, respectively (so $\ell_i$ now replaces notation $m_{2i}$). We denote the longitudinal curves from left to right by $M_k$, $M_{k-1}$, $\ldots$, $M_1$, $M_0$, and renumber their surgery coefficients to be $1/m_k$, $1/m_{k-1}$, $\ldots$, $1/m_1$, $1/m_0$. We denote the $\delta$ curves from left to right by $O_k$, $O_{k-1}$, $\ldots$, $O_1$, $O_0$, with surgery coefficients $0$. This decreasing order of subscripts is suggested by the usual inductive description of a rational tangle and the associated continued fraction $[a_0,\,a_1,\,\ldots,\,a_n] = a_0+1/(a_1+1/(a_2 +\cdots +1/a_n))$, where the coefficient $a_n$ represents the first twist made in the construction and $a_0$ represents the last twist.

We now prove the following theorem.

\begin{theorem}\label{thm-qisS3}
For the general bi-twisted face-pairing given in subsection~\ref{sec-facepairing}, the quotient $3$-manifold is $\Sth$. The image of the north-south axis in $\Sth$ is the two-bridge knot which is the numerator closure of the tangle $T(a/b)$ whose rational number invariant $a/b$ is $[2\ell_0, 2m_0, 2\ell_1, 2m_1, \dots, 2\ell_k, 2m_k]$, or in continued fraction form:
 \begin{equation*}
2\cdot \ell_0+\cfrac{1}{2\cdot m_0+\cfrac{1}{2\cdot \ell_1+
\cfrac{1}{2\cdot m_1+ \cfrac{1}{\ddots}}}}\,.
  \end{equation*}
Here $\ell_0, \ell_1, \ell_2,...$ are the multipliers of the latitudinal 
edge cycles, and $m_0, m_1, m_2,...$ are the multipliers of the 
longitudinal edge cycles.
\end{theorem}

\begin{proof}
We shall reduce the surgery diagram to the empty diagram by a sequence of Rolfsen twists. This will show that the quotient manifold is $\Sth$. We shall track the development of the axis as we perform those twists and show that, at each stage, the knot is a two-bridge knot.  We perform the Rolfsen twists on curves in decreasing order of subscripts in the following order: $M_k$, $L_k$, $O_k$, $M_{k-1}$, $L_{k-1}$, $O_{k-1}$, etc., in order to change surgery coefficients one after the other to $\infty$. Once a coefficient is $\infty$, that curve can be removed from the diagram. 

There are two cases.

Case 1: If $m_k=0$, so that $1/m_k=\infty$, we simply remove $M_k$ and the axis is not affected. We may then remove $L_k$ and $O_k$ without affecting the rest of the diagram as follows. First, twist $-\ell_k=\mp 1$ about $L_k$, to give $L_k$ a surgery coefficient of $\infty$. This allows us to remove $L_k$. It also links $O_k$ and $O_{k-1}$ and changes the surgery coefficient on each from $0$ to $-\ell_k$, but it does not affect the axis or the other link components. Now twist $\ell_k$ times about $O_k$. This allows us to remove $O_k$, returns the surgery coefficient of $O_{k-1}$ to $0$, and leaves the rest of the diagram unchanged. The diagram is now as in Figure~\ref{fig-tbf17}, only with fewer link components. Thus we repeat the argument with this new link component. By induction, either all $m_j=0$, all link components can be removed, resulting in $\Sth$ with the unknot as the image of the axis, or eventually we are in case~2.

Case 2: If $m_k\neq 0$, we twist $(-m_k)$ times about $M_k$. The coefficient of $M_k$ then becomes $\infty$ so that $M_k$ can be removed from the surgery diagram.  This twists two strands of the axis together as in Figure~\ref{fig-tbf18}, introducing $-2\cdot m_k$ half twists into the axis (according to our sign convention). This twist has no effect on the other curves in the diagram.

\begin{figure}
  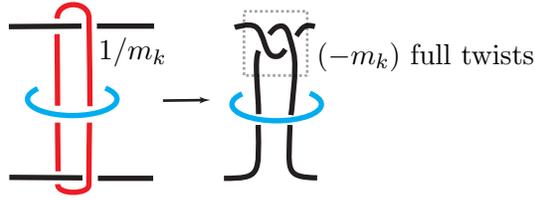
\caption{Removing the curve $M_k$ adds $-m_k$ horizontal twists.}
\label{fig-tbf18}
\end{figure}

Note that the axis has formed a rational tangle at the top-left of the diagram. To identify the tangle, we will use work of Kauffman and Lambropoulou \cite{Kau02}, with attention to orientation. Our twisting orientation agrees with theirs for horizontal twists, and so at this point, the rational tangle has continued fraction with the single entry $[-2m_k]$. 

The proof now procedes by induction. We will assume that at the $j$-th step, we have a surgery diagram with image of the axis with the following properties. 
\begin{enumerate}
\item In the top left corner, there is a rational tangle $T_j$ with continued fraction
  \[
  [-2m_j, -2\ell_j, \dots, -2\ell_k, -2m_k].
  \]
\item Two strands run from the tangle through the link component $O_j$.
\item Link components $M_k$, $L_k$, $\dots$, through $M_j$ have been removed. 
\item To the right, the surgery diagram is identical to the original surgery diagram, beginning with link components $L_j$ and running to the right through the components $M_0$ and $L_0$. That is, the link components are identical for this portion of the diagram, and the surgery coefficients are also identical.
\end{enumerate}

\begin{figure}
  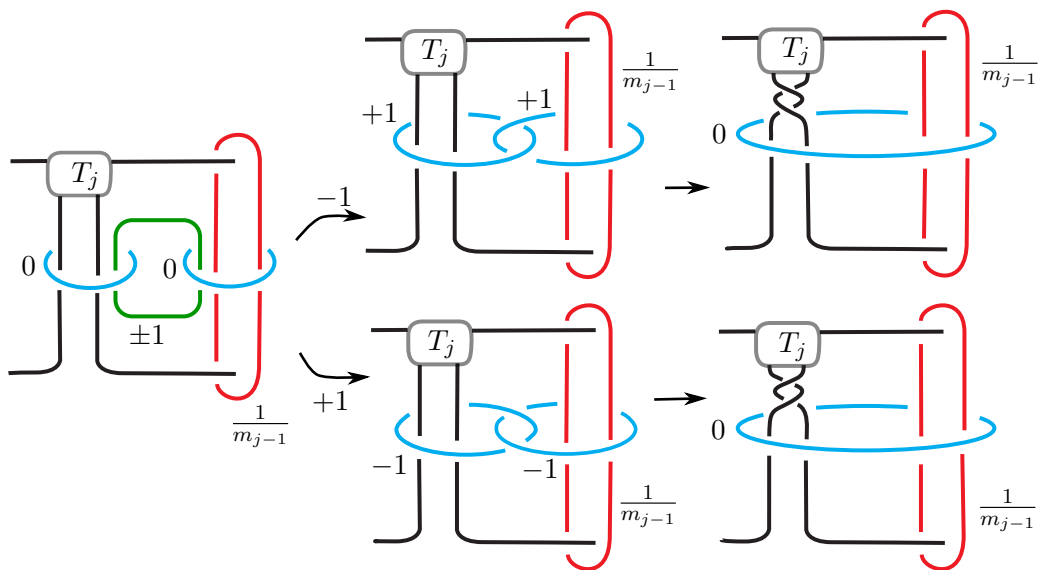
\caption{The effect of Rolfsen twists to remove first $L_j$ and then $O_j$.}
\label{fig-tbf19}
\end{figure}

The next step is to remove link components $L_j$ and $O_j$. This is shown in Figure~\ref{fig-tbf19}, for both cases $\ell_j=\pm 1$. Carefully, we twist $-\ell_j$ times about $L_j$. The coefficient of $L_j$ then becomes $\infty$ so that $L_j$ can be removed from the surgery diagram. That twist adds $-\ell_j$ to the $0$ surgery coefficients of $O_j$ and $O_{j-1}$ and links those two curves together with overcrossing having sign equal to $-\ell_j$.  This twist has no effect on the axis. Now twist $\ell_j$ times about $O_j$. The coefficient of $O_j$ then becomes $\infty$ so that $O_j$ can be removed from the surgery diagram. The twist returns the surgery coefficient of $O_{j-1}$ back to $0$. The twist also adds $2\cdot \ell_j$ half twists to the two strands of the axis that were running through $O_j$. Note this yields a new rational tangle, with a vertical twist added to the tangle $T_j$. Our twisting orientation for vertical twists is opposite that of Kauffman and Lambropoulou \cite{Kau02}, and so the continued fraction of this new tangle becomes $T=[-2\ell_j, -2m_j, \dots, -2\ell_k, -2m_k]$.

\begin{figure}
  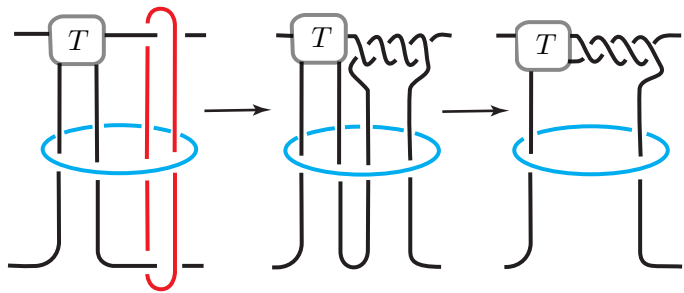
\caption{Removing $M_{j-1}$ through twisting.}
\label{fig-tbf20}
\end{figure}

We now need to consider $M_{j-1}$. If $m_{j-1}=0$, so its surgery coefficient is $\infty$, we simply remove $M_{j-1}$ from the surgery diagram, and we have completed the inductive step. Otherwise, we twist $-m_i$ times about $M_i$, as in Figure~\ref{fig-tbf20}, after which four strands of the axis pass through $O_{j-1}$. However, the central two strands can be isotoped upward through $O_{j-1}$. This adds $-2m_{j-1}$ horizontal crossings to the tangle $T$, yielding a tangle $T_{j-1}$, and completes the inductive step. 

After the final step $j=0$, we have removed all $M_j$, $L_j$, $O_j$ from the surgery diagram, yielding $\Sth$, and our axis has become the denominator closure of a rational tangle $T(c/d)$ with continued fraction
\[
[-2\ell_0, -2m_0,\dots, -2\ell_k, -2m_k] = 
\cfrac{1}{-2\ell _0+\cfrac{1}{-2 m_0+\cfrac{1}{-2 \ell_1+
\cfrac{1}{-2 m_1+ \cfrac{1}{\ddots}}}}}\,.
\]
The continued fraction begins with $1/(-2\ell _0+\cdots)$ instead of $-2\cdot \ell_0+\cdots $ because $\ell _0$ corresponds to a vertical twist. Loosely speaking, horizontal twists correspond to addition and vertical twists correspond to addition and inversion. Hence our knot is the numerator closure of the tangle $T(a/b)$ with $a/b=-d/c$, as in the statement of the theorem.
\end{proof}

Recall from the introduction that a multiplier function $m$ with values $m_0,\dotsc,m_{2k+1}$ is normalized if $m_{2k+1}\ne 0$ and if $m_{2i+1}=0$ for some $i\in\{0,\ldots,k-1\}$, then $m_{2i}=m_{2i+2}$. The following example helps to motivate this definition.

\begin{example}
Figure~\ref{fig-tbf21} shows an example arising from multipliers given as follows:
\[
m_6=3,\,m_5=0,\,m_4 = 0,\,m_3 = 2,\,m_2=-3,\,m_1=0,\,m_0=2;
\]
\[
\ell_6 =1 ,\,\ell_5 = -1,\,\ell_4 = -1,\,\ell_ 3=-1 ,\,\ell_2 =1
,\,\ell_1 =1 ,\,\ell_0 =1 .
\]
In the notation of the previous paragraph, the multiplier function has values $\ell_0$, $m_0$, $\ell_1$, $m_1$, $\dots$, $m_6$.  This multiplier function is not normalized since $\ell _5=-\ell _6$ even though $m_5=0$.  As a result, the second vertical twist cancels the first one, and so they can be eliminated.  This is consistent with the fact that $x+ 1/(0+1/y)=x+y$, so that a continued fraction with a term equal to 0 can be simplified.  Also notice that if $m_6=0$ instead of $m_6=3$, then the first three vertical twists can be untwisted, and so they can be eliminated.  This is consistent with the fact that
$x+\frac{1}{y+1/0}=x$.
\end{example}

\begin{figure}
  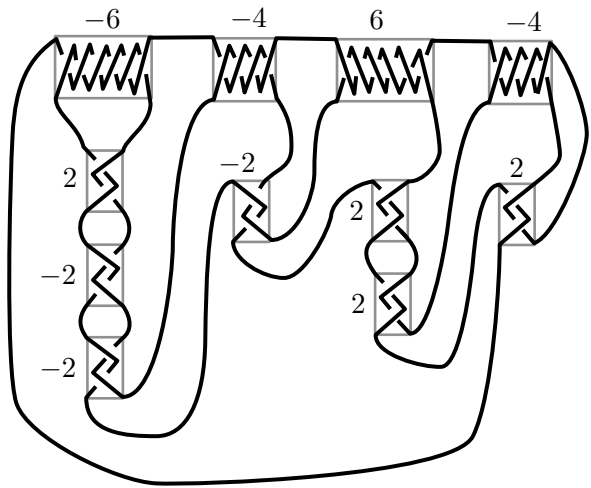
\caption{An example.}
\label{fig-tbf21}
\end{figure}

\begin{corollary}\label{cor-denomclosure}
Every normalized multiplier function yields a nontrivial two-bridge knot.  Every nontrivial two-bridge knot $K$ is realized by either one or two normalized multiplier functions.  Furthermore, if $K$ is the numerator closure of the tangle $T(a/b)$, then it has exactly one such realization if and only if $b^2\equiv 1 \text{ mod } a.$
\end{corollary}

\begin{proof}
Note that our construction allows us to obtain any two-bridge knot with a rational invariant made only of even integers, by choosing $m_j=0$ appropriately. On the other hand, it is a classical result that any rational number $p/q$ with $p$ odd and $q$ even has a continued fraction expansion of the form $[2a_0,\dotsc,2a_n]$ with $n$ odd.  This result can also be derived by a modification of the Euclidean algorithm.  The corollary then follows from Theorem~\ref{thm-qisS3} and standard results involving two-bridge knots, much of which is contained in \cite{Ble88} and \cite{Kau02}.
\end{proof}

\section{Cyclic Presentations}\label{sec-cyclicpresentations}

Let $M_n(K_m)$ denote the $n$-fold branched cyclic covering of $\Sth$, branched over the two-bridge knot $K_m$ realized by the multiplier $m$. It is known (see~\cite{Cav99}) that the fundamental group $G_n$ of $M_n(K_m)$ has a cyclic presentation. We shall show here that the bi-twist representation of $M_n(K_m)$ easily leads to the same result.

\begin{definition}
Let $X = \{x_1,\ldots, x_{n}\}$ be a finite alphabet.   Let $\phi$ denote the cyclic permutation of $X$ that takes each $x_i$ to $x_{i+1}$, with subscripts taken modulo $n$. Let $W(X)$ denote a finite word in the letters of $X$ and their inverses. Then the group presentation
\[
\langle X\mid W(X),\,\phi(W(X)),\,\ldots, \phi^{n-1}(W(X))\rangle
\]
is called a \emph{cyclic presentation}.
\end{definition}

\begin{theorem}\label{thm-Gn}
The group $G_n = \pi_1(M_n(K_m))$ has a cyclic presentation.
\end{theorem}

Before giving the proof, we recall the algorithm that gives a presentation for the fundamental group of the bi-twist manifold $M(\zep, m)$. We work with the model faceted 3-ball. We assign a generator $x(f)$ to each face $f$. We will need to assign a word $W(f,e)$ to each pair $(f,e)$ consisting of a face $f$ and boundary edge $e$ of $f$, and a word $W(f)$ to each face $f$. 

If $f$ is a face, denote the matching face by $f^{-1}$. Then $x(f^{-1}) = x(f)\inv$. If $f$ is a face and $e$ is a boundary edge of $f$, then there is a
(shortest) finite sequence $(f,e)=(f_1,e_1)$, $(f_2,e_2)$, $\ldots$, $(f_k,e_k)= (f,e)$ such that $\zep(f_i)$ takes $e_i$ onto $e_{i+1}$ and takes $f_i$ onto the face across $e_{i+1}$ from $f_{i+1}$. We define $W(f,e) $ to be the word $ x(f_1)\cdot x(f_2)\cdots x(f_{k-1})$. Finally, if $f$ is a face and $e_1$, $e_2$, $\ldots$, $e_j$ are the edges of $f$, in order, with assigned multipliers $m_1$, $m_2$, $\ldots$, $m_j$, then we assign $f$ the word
\[
W(f) = W(f,e_1)^{m_1}\cdot W(f,e_2)^{m_2}\cdots W(f,e_j)^{m_j}.
\]
The following lemma follows from standard results. See also \cite[Theorem~4.8]{CFP2}. 

\begin{lemma}\label{lemma:Presentation1}
The group $\pi_1(M(\zep,m))$ has presentation
\[
\langle x(f), \text{ $f$ a face }\mid W(f), \text{ $f$ a face}\rangle \]
\end{lemma}

\begin{proof}[Proof of Theorem~\ref{thm-Gn}]
Begin with a model faceted 3-ball and multipliers $\ell_0, m_0, \dotsc, \ell_k, m_k$ used to construct $M_1(K_m)$ in Section~\ref{sec-generalsurgery}.  We take its $n$-fold branched cyclic cover branched over the north-south axis.  We label the faces of the northern hemisphere $x(i,j)$ as in Figure~\ref{fig-tbf22}.

\begin{figure}
  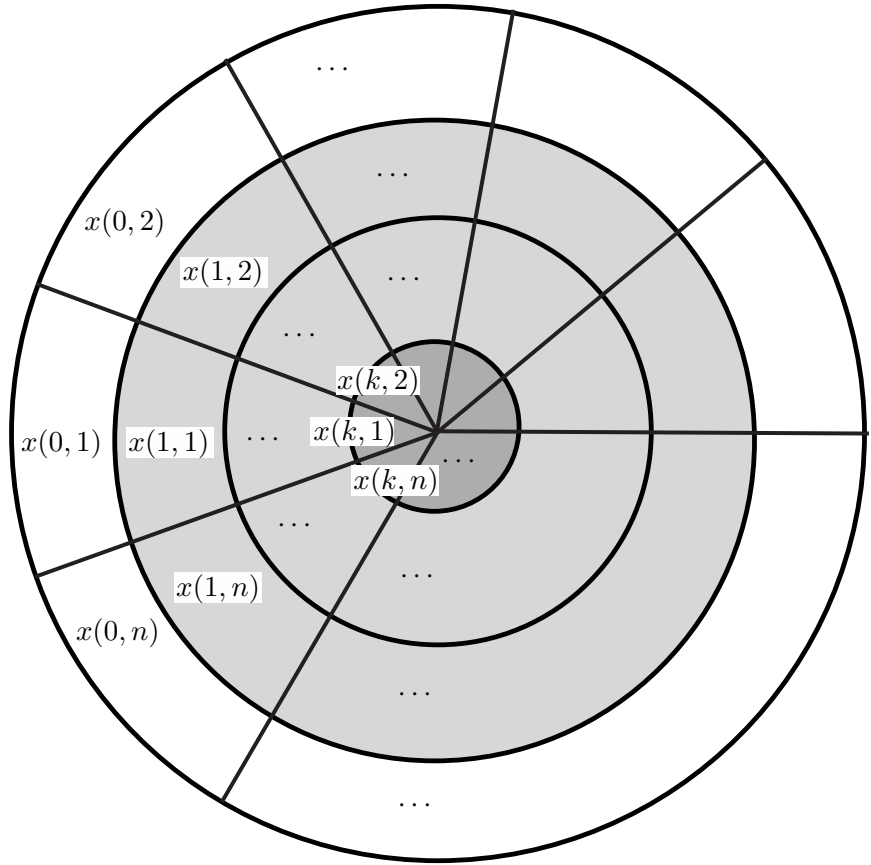
\caption{The model for the $n$-fold branched cyclic cover, with the face
generators labeled $x(i,j)$. Faces of type $0$ are shaded white, faces of type $1$ are shaded light gray, and faces of type $2$ are shaded darker gray.}
\label{fig-tbf22}
\end{figure}

\begin{figure}
  \begin{tabular}{ccc}
    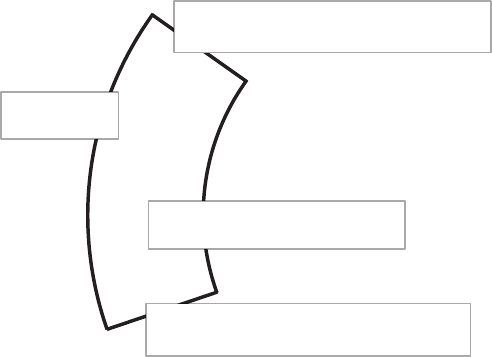 &
    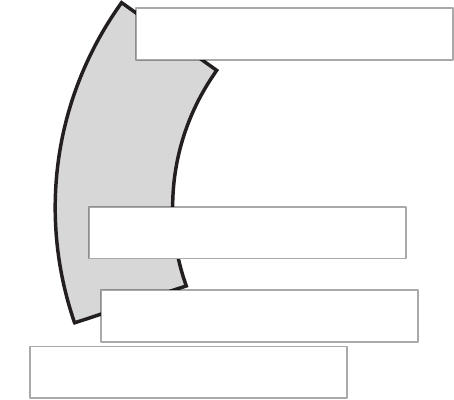 &
    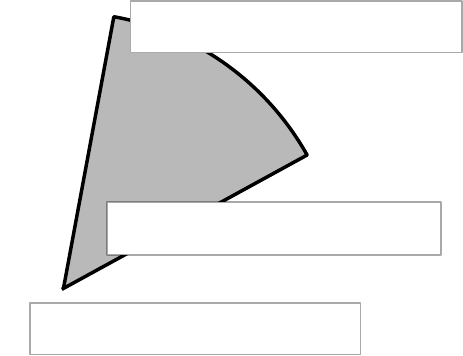 \\
    (a) & (b) & (c) \\
  \end{tabular}
\caption{(a) A face of type $0$, with face-edge words. (b) A face of type $1$. (c) A face of type $2$.}
\label{fig-tbf23}
\end{figure}

We use the same labels $x(i,j)$ as group generators. The corresponding faces and generators for the southern hemisphere are $x(i,j)\inv$. We distinguish three types of faces: those bordering on the equator, which are designated as type 0, those touching the poles, which are designated as type 2, and all others, designated type 1.  We initially assume that $k>0$ so that we don't have faces that are both type 0 and type 2.  Since edge classes have size 1 or size 2, the words associated with a face-edge pair have length 1 or length 2.
Figure~\ref{fig-tbf23} 
shows edges of the three types of faces labeled with those face-edge words. These words are then raised to the appropriate powers and multiplied together to give the word associated with the corresponding face. We call these words $R(i,j)$'s since they are the relators of the fundamental group.
\begin{align*}
R(0,j) &= [x(0,j)]^{\ell_0} [x(0,j)x(0,j+1)^{-1}]^{m_0} [x(0,j)x(1,j)^{-1}]^{\ell_1} [x(0,j)x(0,j-1)^{-1}]^{m_0} \\
R(i,j) &= [x(i,j)x(i-1,j)^{-1}]^{\ell_i} [x(i,j)x(i,j+1)^{-1}]^{m_i}
[x(i,j)x(i+1,j)^{-1}]^{\ell_{i+1}} [x(i,j)x(i,j-1)^{-1}]^{m_i} \\
R(k,j) &= [x(k,j)x(k-1,j)^{-1}]^{\ell_k} [x(k,j)x(k,j+1)^{-1}]^{m_k}
[x(k,j)x(k,j-1)^{-1}]^{m_k}
\end{align*}

We conclude that the fundamental group has a presentation
\[
\langle x(i,j) \mid R(i,j),\, i=0,\ldots,k;\,j=1,\ldots, n\rangle.
\]
Since each of the multipliers $\ell_0,\ell_1,\dots,\ell_k$ is either $+1$ or $-1$, the letter $x(1,j)^{\pm1}$ appears at most once in the relator $R(0,j)$. Similarly, the letter $x(i,j)^{\pm1}$ appears at most once in the relator $R(i-1,j)$, for $i=2, \dots, k-1$, and the letter $x(k,j)^{\pm1}$ appears at most once in the relator $R(k-1,j)$. Hence, these relators may be solved for $x(1,j), x(2,j), \ldots, x(k,j)$ iteratively, and then these relators and generators may be removed. The only generators remaining are the generators $x(0,j)$, with $j=1,\ldots,n$; and, with appropriate generator substitutions made, the only remaining relators are the relators $R(k,j)$. The presentation
\[
\langle x(0,j)\mid R(k,j),\, j=1,\ldots, n\rangle
\]
is clearly a cyclic presentation.

Finally, if $k=0$, then every face is both type 0 and type 2. In this case the presentation is $\langle x(0,j) | R(0,j), j=1,\dots,n \rangle$, which is cyclic.
\end{proof}

\subsection{The Fibonacci and Sieradski manifolds}\label{sec:fibsier}

We illustrate these group calculations with the Fibonacci manifolds and the Sieradski manifolds. We prove the following theorem. 

\begin{theorem}\label{thm:FnSnPresentation}
The fundamental group $\pi_1(F_n)$ is the $2n$-th Fibonacci group with presentation
\[
\langle x_1,\ldots,x_{2n}\mid x_1x_2=x_3,\,x_2x_3=x_4,\,\ldots,\,x_{2n-1}x_{2n}=x_1,\,x_{2n}x_1=x_2\rangle.
\]
The fundamental group $\pi_1(S_n)$ is the $n$-th Sieradski group with presentation
\[
\langle y_1,\ldots,y_n\mid y_1 = y_2y_n,\,y_2 =
y_3y_1,\,y_3=y_4y_2,\,\ldots,\,y_n = y_1y_{n-1}\rangle.
\]
\end{theorem}

\begin{proof}
The faceted 3-ball that serves as the model for the face pairings is the same for both manifolds; it is as in Figure~\ref{fig-tbf22} with $k=0$, so without interior latitudinal circles.

For the Fibonacci manifolds, we label the faces of the northern hemisphere as $x(2)$, $x(4)$, $\ldots$, $x(2n)$. All subscript calculations are modulo $2n$. We obtain the following cyclic presentation for the fundamental group:
\begin{align*}
\langle x(2),\,x(4),\,\ldots,\,x(2n)\mid 
x(2j)\cdot [x(2j)x(2j+2)\inv]\cdot [x(2j)x(2j-2)\inv], \, j=1,2,\ldots,\,n\rangle.
\end{align*}
We can then introduce intermediate generators $x(2j-1) = x(2j-2)\inv\cdot x(2j)$. The presentation becomes the standard presentation for the $2n$-th
Fibonacci group, as desired:
\[
\langle x(1),\,\ldots,\,x(2n)\mid x(i+2)=x(i)\cdot x(i+1) \rangle.
\]

For the Sieradski manifolds, we label the faces of the northern hemisphere as $y(1)$, $y(2)$, $\ldots$, $y(n)$. Subscript calculations are modulo $n$. We obtain the following cyclic presentation for the fundamental group:
\[
\langle y(1),\ldots,y(n)\mid y(j)\inv\cdot [y(j)y(j+1)\inv] \cdot [y(j)y(j-1)\inv], j=1,\dotsc,n\rangle,
\]
or, reversing the order of the subscripts so that $x(1) = y(n),\,\ldots,\, x(n) = y(1)$,
\[
\langle x(1),\ldots, x(n)\mid  x(i) = x(i-1)\cdot x(i+1)\rangle,
\]
the standard presentation for the $n$th Sieradski group.
\end{proof}

\subsection{Branched cyclic covers with periodic homology}

In this section we consider first homology groups of our cyclic branched covers of $\Sth$.  This is a topic which has received and still receives considerable attention.  There are two very different behaviors.  The first homology groups of the $n$-fold cyclic covers $M_n$ of $\Sth$ branched over a knot $K$ are either periodic in $n$ or their orders grow exponentially fast.  Specifically, Gordon~\cite{Gor72} proved that when the roots of the Alexander polynomial of $K$ are all roots of unity, then $\text{H}_1(M_n,\mathbb{Z})$ is
periodic in $n$.  Riley~\cite{Ril90} and, independently, Gonzalez-Acuna and Short~\cite{Gon91} proved that if the roots of the Alexander polynomial are not all roots of unity, then the finite values of $\text{H}_1(M_n,\mathbb{Z})$ grow exponentially fast in $n$. Silver and Williams~\cite{Sil02} extended these results to links and replaced ``finite values'' with ``orders of torsion subgroups''.  See also Le~\cite{Le09}, Bergeron and Venkatesh~\cite{Ber13} and Brock and Dunfield~\cite{Bro13} for more recent results and conjectures on this
topic.

We are particularly fascinated by the first homology of the branched cyclic covers of $\Sth$ branched over the knots that are  two-strand braids. These knots are the only two-bridge knots that are not hyperbolic.

The northern hemisphere of the model before bi-twisting looks like
Figure~\ref{fig-tbf22}.  We construct the $n$-fold branched cyclic cover of $\Sth$, branched over a knot that is a $2$-strand braid, by using $k\ge 0$ latitudes and $n$ longitudes in the open northern hemisphere, assigning multipliers $-1$ to the latitudinal edges, and assigning multipliers $+1$ to all longitudinal edges. We calculate the fundamental group as in the proof of Theorem~\ref{thm-Gn} and transform it into a cyclic presentation as explained there. We then abelianize, and let $a_0,\,a_1,\,\ldots,\, a_{2k+2}$ denote the
exponent sums of the generators in the defining cyclic word $W$.

We very briefly indicate by diagram how these integers may be computed.  Every relator corresponds to a diagram as follows.
\begin{equation*}
\begin{array}{cc|ccc}
 & & j-1 & j & j+1 \\ \cline{2-5}
R(k,j) &  k & 1 & -1 & 1 \\
 & k-1 & & -1 &
\end{array}
  \end{equation*}
  \begin{equation*}
\begin{array}{cc|ccc}
 & & j-1 & j & j+1 \\ \cline{2-5}
 & i+1 &  & -1 &  \\
R(i,j) &  i & 1 & 0 & 1 \\
 & i-1 & & -1 &
\end{array}
  \end{equation*}
  \begin{equation*}
\begin{array}{cc|ccc}
 & & j-1 & j & j+1 \\ \cline{2-5}
 & 1 &  & -1 &  \\
R(0,j) &  0 & 1 & 0 & 1 
\end{array}
  \end{equation*}
We begin with the diagram for $R(k,j)$ and use the diagrams for $R(k-1,j)$, $R(k-2,j)$, $\ldots$ to successively transform the entries in rows $k$, $k-1$, $\ldots, 1$ to 0.  The defining cyclic word is the final result in row 0.

\begin{equation*}
\begin{aligned}
\begin{matrix}  1 & -1 & 1 \\
 & -1 &
\end{matrix} &
\longrightarrow 
\begin{matrix} & 0 & 0 & 0 & \\
1 & -1 & -1+1+1 & -1 & 1\\
  & -1 & 1 & -1 &
\end{matrix}\\ &
\phantom{x}\\ &
\longrightarrow 
\begin{matrix} & & 0 & 0 & 0 & &\\
  & 0 & 0 & 0 & 0 & 0 & \\
1 & -1 & -1+1+1 & 1-1+1 & -1+1+1 & -1 & 1\\
  & -1 & 1 & -1 & 1 & -1
\end{matrix}\\ &
\longrightarrow \quad\ldots
\end{aligned}
\end{equation*}
We find that the polynomial $a_0+a_1\cdot t+ \cdots + a_{2k+2}\cdot t^{2k+2}$ is the cyclotomic polynomial
\[
1 - t + t^2 -t^3 + \cdots -t^{2k+1}+t^{2k+2}.
\]
(If $2k+3 > n$, then the polynomial folds on itself because powers are to be identified modulo $n$. However, once $n\ge 2k+3$, there is no folding.) 

\begin{remark}
The computation indicated by diagram is a continued fraction algorithm.  For the fundamental group of a general two-bridge knot, the corresponding polynomial may be taken to be the numerator of the continued fraction
\begin{equation*}
Q_0-\cfrac{1}{Q_1-\cfrac{1}{Q_2-\cfrac{1}{\ddots- \cfrac{1}{Q_k}}}}
\end{equation*}
where
\begin{equation*}
Q_i(t) = m_it-(\ell _i+\ell _{i+1}+2m_i)+m_it^{-1}\quad\text{for }0\le i\le k-1
\end{equation*}
and
\begin{equation*}
Q_k(t) = m_kt-(\ell _k+2m_k)+m_kt^{-1}.
\end{equation*}
\end{remark}

We shall prove that, for a given knot realized as a two-strand braid, the abelianizations of the fundamental group of the $n$-fold branched cover are periodic functions of $n$. However, as a warm up, we use row reduction of the presentation matrix to prove the much easier theorem that no two of the Fibonacci groups $F(n)$ are isomorphic for $n>1$ since no two of the abelianizations have the same order. This problem appears as an exercise on page $35$ of~\cite{Joh76}, where Johnson suggests using the two-variable presentation of the group. We use the $n$-variable presentation and note that the Fibonacci numbers $f_0=0, f_1=1, f_2=1, f_3=2, \ldots$ appear in a very natural way.  In this case we have the behavior of exponential growth of orders.

\begin{theorem}\label{thm:Fibonacci-growth}
Let $F(n) = \langle x_1,\,x_2,\,\ldots,\,x_n\mid x_ix_{i+1}=x_{i+2} \text{ for all $i$}\rangle$, with subscripts calculated modulo $n$. For odd $n$, the order of the abelianization is the sum $f_{n-1} + f_{n+1}$ of two Fibonacci numbers. For even $n$, the order is $f_{n-1} + f_{n+1} - 2$.
\end{theorem}

\begin{remark}
Recall that for even $n$ these abelianizations are the first homology groups of the Fibonacci manifolds. This theorem gives successive orders of  $1$, $1$, $4$, $5$, $11$, $16$, $29$, $45$, $76$, $121$, $\ldots$ for the abelianizations of the Fibonacci groups. It is clear from the definition of the Fibonacci numbers that these numbers are strictly increasing after the numbers $1,1$. These numbers are also known as the \emph{Associated Mersenne numbers} A001350 in Sloane's ``The On-Line Encyclopedia of Integer Sequences''. The sums $f_{n-1} + f_{n+1}$ are also known as Lucas numbers. 
\end{remark}

\begin{proof} The presentation matrix has the following form:
\begin{equation*}
\begin{pmatrix} 1&1&-1&0&0&\cdots&0&0&0\\
                                 0&1&1&-1&0&\cdots&0&0&0\\
                                 0&0&1&1&-1&\cdots&0&0&0\\
                                   &  & \cdots &  &   &\cdots
&&\cdots&\\
                                 0&0&0&0&0&\cdots&1&1&-1\\
                                 -1&0&0&0&0&\cdots&0&1&1\\
                                 1&-1&0&0&0&\cdots&0&0&1\\
\end{pmatrix}
\end{equation*}
The absolute value of the determinant of this matrix is the order of the abelianization of the group unless the determinant is $0$. In that case, the group is infinite.  The goal is to move the entries in the lower left corner to the right by adding multiples of the upper rows. These operations do not change the determinant.

We use the upper rows in descending order, with each successive row moving the lower-left $2\times 2$ matrix one column to the right. We first trace the evolution of the two entries in the next-to-last row:
\[
(-1,0)\to (1,-1)\to (-2,1)\to(3,-2)\to (-5,3)\to (8,-5)\to \cdots.
\]
The reader will easily identify the first in the $k$th pair as $(-1)^k f_k$, and the second as $(-1)^{k-1} f_{k-1}$. Since the second of these, namely $(1,-1)$, coincides with the first pair in the bottom row, we see that the bottom row evolves just one step ahead of the next-to-last row. Thus after $k$ moves, the $2\times 2$ matrix evolves into the matrix
\[
\begin{pmatrix} (-1)^k f_k& (-1)^{k-1}f_{k-1}\\
                               (-1)^{k+1}f_{k+1}&(-1)^k f_k 
\end{pmatrix}\]
which has determinant $f_k^2 - f_{k+1}\cdot f_{k-1} =(-1)^{k-1}$. After the appropriate number of moves, this matrix will be  added to the matrix 
\[
\begin{pmatrix} 1&1\\
0&1
\end{pmatrix}
\]
from the lower right corner to form the very last lower right corner matrix
\[
\begin{pmatrix} (-1)^k f_k+1& (-1)^{k-1}f_{k-1}+1\\
                               (-1)^{k+1}f_{k+1}&(-1)^k f_k +1
\end{pmatrix}.
\]
The matrix then has determinant
\begin{equation*}
\begin{aligned}[]
[f_k^2 +2\cdot(-1)^k\cdot f_k + 1] & -[f_{k+1}\cdot f_{k-1}
+(-1)^{k+1} f_{k+1}]\\
 & = (-1)^{k+1} + 1 + (-1)^k[f_k + f_{k+2}].
\end{aligned}
\end{equation*}
The absolute value of this determinant is the order of the abelianization, and since the last value of $k$ is $n-1$, it agrees with the value claimed in the theorem.
\end{proof}

For the moment, we fix two integers $j>0$ and $k\ge 0$, and let $G_n$, with $n = j + 1 + k$, denote an Abelian group with generators $x_0, x_1, x_2, \ldots$ such that $x_i = x_{i+n}$ and with relators $a_0\cdot x_i + a_1\cdot x_{i+1} + \cdots + a_j\cdot x_{i+j}$ for each $i$. Then the group has a circulant relator matrix of the form shown in Figure~\ref{fig-tbf28}.  In the following theorem we have the behavior of periodic homology groups.

\begin{figure}[ht]
  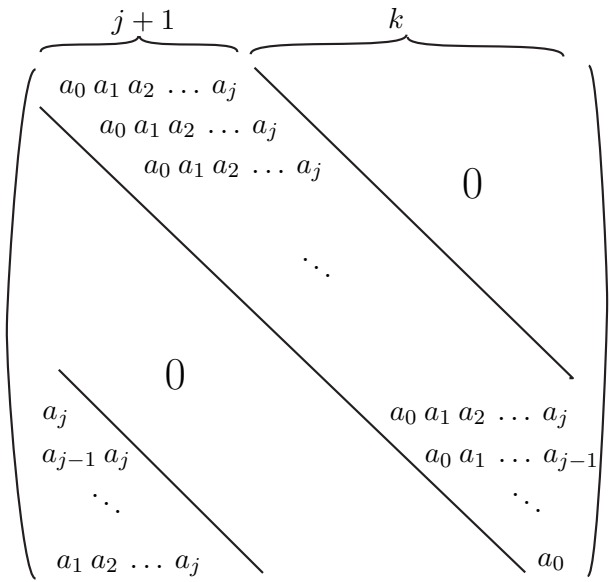
\caption{The relator matrix for $n=j+1+k$.}
\label{fig-tbf28}
\end{figure}

\begin{theorem}
Let $j$, $k$, and $G_n$ be as immediately above, so that $n = j + 1 + k$. Assume that $p(t) = a_0 + a_1\cdot t + \cdots + a_j\cdot t^j$ is a cyclotomic polynomial, by which we mean that there is a polynomial $q(t) = b_0 + b_1\cdot t + \cdots + b_\ell\cdot t^\ell$ such that $p(t)\cdot q(t) = 1-t^{j+\ell}$.  Then the groups $G_n$ and $G_{n+j + \ell}$ are isomorphic.
\end{theorem}

\begin{proof}
We manipulate the relator matrix for $G_{n+j+\ell}$ using integral row and column operations. See Figure~\ref{fig-tbf28}.  We use the rows at the top of the matrix to remove entries from the triangle at the lower left corner of
the matrix.

Let $x$ be such an entry in row $R_a$. Let $R_b$ denote the row whose initial entry on the diagonal is above $x$. Subtract from row $R_a$ the sum $x\cdot [ b_0 \cdot R_b + b_1\cdot R_{b+1} + \cdots + b_\ell \cdot R_{b+\ell}]$. The effect is to move entry $x$ to the right $j+\ell$ places. Similarly, we move all entries in the lower left triangle $j+\ell$ places to the right. 
Because $a_0=\pm 1$, we may use column operations to make every entry to the right of the first $j+\ell $ $a_0$'s equal to 0.  The lower right $n\times n$ block of the resulting matrix is the relator matrix for $G_n$. The theorem follows.
\end{proof}

\begin{remark}
The same calculation can be carried out if the polynomial is any integer multiple $\za \cdot p(t)$ of a cyclotomic polynomial $p(t)$, except that the diagonal entries above the periodic box all become $\za$'s. Thus the abelianization has a periodic component together with an increasing direct sum of $\Z_\za$'s.  It can be shown that these are the only polynomials with these
periodicity properties.
\end{remark}

\begin{corollary}\label{cor:2strandBraidPeriodic}
If $K$ is a knot that is a 2-strand braid and $M_n$ is the $n$-fold cyclic branched cover of $\Sth$ over $K$, then the homology groups $H_1(M_n)$ are periodic in $n$.
\end{corollary}

\begin{remark}
Lambert, in his Ph.D.\ dissertation at Brigham Young University~\cite{Lam10}, explicitly calculated all of the homology groups of the branched cyclic covers of $\Sth$, branched over knots that are $2$-strand braids. These are the only two-bridge knots that are not hyperbolic. His tables give an explicit picture of the periodicity we have just proved. Rolfsen~\cite{Rol76} notes that the period for the trefoil is $6$. We shall also see that as follows.
\end{remark}

\begin{proof}[Proof of Corollary~\ref{cor:2strandBraidPeriodic}]
It suffices to find the appropriate polynomials $q(t)$, and thereby determine the period. If $p(t) = 1 - t + t^2$, as for the trefoil, then the appropriate $q(t)$ of smallest degree is $q(t) = 1 + t - t^3 - t^4$ so that the period is $2 + 4 = 6$. With $5$ half twists, $p(t) = 1 - t + t^2 - t^3 + t^4$ and $q(t) = 1 + t - t^5-t^6$ and the period is $4+6=10$. Each added pair of half twists in the braid adds two terms to $p(t)$, multiplies the negative entries of $q(t)$ by $t^2$, and increases the period by $4$. 
\end{proof}

\begin{remark}
By Gordon~\cite{Gor72}, the homology groups $H_1(M_n)$ of the cyclic branched covers $M_n$ of the complement of a knot $K$ are periodic with period dividing $m$ if and only if the first Alexander invariant (the quotient of the first two Alexander polynomials) of $K$ is a divisor of the polynomial $t^m -1$. Furthermore, if the first Alexander invariant is a divisor of $t^m-1$ and $n$ is a positive integer, then $H_1(M_n) = H_1(M_{(m,n)})$, where $(m,n)$ is the greatest common divisor of $m$ and $n$. Since the first Alexander invariant of the trefoil knot is $1 -t + t^2$, which divides $t^6-1$, Gordon's theorem shows that the first homology groups of the cyclic branched covers of the trefoil
knot are periodic with period $6$ and $H_1(S_{6j+2}) = H_1(S_{6j+4})$ for all $j$.
\end{remark}

We may use the calculation of the period of the trefoil in establishing the following theorem.

\begin{theorem}\label{thm:siergps}
No two of the Sieradski groups are isomorphic.  Hence no two of the branched cyclic covers of $\Sth$, branched over the trefoil knot, are homeomorphic.
\end{theorem}

\begin{proof}
In \cite{Mil75}, Milnor defines the Brieskorn manifold $M(p,q,r)$ to be the
orientable closed 3-manifold obtained by intersecting the complex algebraic surface given by $z_1^p+z_2^q+z_3^r=0$ with the unit sphere given by $\left|z_1\right|^2+\left|z_2\right|^2+\left|z_3\right|^2=1$. Here $p$, $q$, $r$ should be integers at least 2.  Theorem 2.1 of~\cite{Cav98} by Cavicchioli, Hegenbarth, and Kim states that $S_n$ is the Brieskorn manifold $M(2,3,n)$. This follows from the fact that $S_n$ is the $n$-fold cyclic branched cover of $S^3$ branched over the trefoil knot, which is the torus knot
of type $(2,3)$, and Lemma 1.1 of~\cite{Mil75}, which states that the Brieskorn manifold $M(p,q,r)$ is the $r$-fold cyclic branched cover of $S^3$ branched over a torus link of type $(p,q)$.

The first few $n$-fold cyclic covers of $S^3$ branched over the right-hand trefoil knot are discussed in Section 10D of Rolfsen's book~\cite{Rol76}, which begins on page 304.  Here are the results.

\begin{itemize}
\item $n=1$: The manifold $S_1$ is the 3-sphere $S^3$, and so $G_1=1$.
\item $n=2$: The manifold $S_2$ is the lens space $L(3,1)$, so $G_2\cong \Z/3\Z$.
\item $n=3$: The manifold $S_3$ is the spherical 3-manifold with fundamental group $G_3$ the quaternion group of order 8.  It appears in Example~7.2 of~\cite{CFP2}.  This group might be called the binary Klein 4-group.
\item $n=4$: The manifold $S_4$ is the spherical 3-manifold with fundamental group $G_4$ the binary tetrahedral group.
\item $n=5$: The manifold $S_5$ is the spherical 3-manifold with fundamental group $G_5$ the binary icosahedral group.  In other words, this is the Poincar\'{e} homology sphere.
\item $n=6$: The manifold $S_6$ is the Heisenberg manifold.  Here
\[
G_6\cong\left\langle x,y:[x,[x,y]]=[y,[x,y]]=1\right\rangle.
\]
\end{itemize}

In~\cite{Mil75} Milnor proves that $M(2,3,n)$, which we know is homeomorphic to $S_n$, is an $\widetilde{SL}(2,\R)$-manifold for $n\ge 7$. It follows that $G_1,\dotsc,G_6$ are distinct and that they are not $\widetilde{SL}(2,\R)$ manifold groups.  Because of this and Milnor's result that $S_n$ is an $\widetilde{SL}(2,\R)$-manifold for $n\ge 7$, to prove that the groups $G_n$ are distinct, it suffices to prove that the groups $G_n$ are distinct for $n\ge 7$.

As stated on page 304 of Rolfsen's book \cite{Rol76}, for every positive integer $n$ the first homology group $H_1(S_n)$ is $\Z\oplus \Z$, 0, $\Z/3\Z$ or $\Z/2\Z\oplus \Z/2\Z$ when $n\equiv 0,\pm 1,\pm 2 \text{ or }3 \text{ mod } 6$.  So to prove that Sieradski groups $G_m$ and $G_n$ are distinct, we may assume that $m\equiv\pm n\text{ mod } 6$.

For the rest of this section suppose that $n\ge 7$.  In~\cite{Mil75} (see the bottom of page 213 and Lemma 3.1) Milnor proves that $G_n$ is isomorphic to the commutator subgroup of the centrally extended triangle group
$\Zc(2,3,n)=\left\langle\zc_1,\zc_2,\zc_3:\zc_1^2=\zc_2^3=\zc_3^n=\zc_1\zc_2 \zc_3\right\rangle$.

Let $\Zd(2,3,n)=\left\langle\zd_1,\zd_2,\zd_3:\zd_1^2=\zd_2^3=\zd_3^n= \zd_1\zd_2\zd_3=1\right\rangle$, a homomorphic image of $\Zc(2,3,n)$.  The group $\Zd(2,3,n)$ is the group of orientation-preserving elements of the $(2,3,n)$-triangle group. Let $\Zd'(2,3,n)$ denote the commutator
subgroup of $\Zd(2,3,n)$.  We see that  the quotient group $\Zd(2,3,n)/\Zd'(2,3,n)$ is isomorphic to the group generated by the elements $(1,0,0)$, $(0,1,0)$, and $(0,0,1)$ in $\Z^3$ with relations corresponding to a matrix which row reduces as follows.
\begin{equation*}
\left[\begin{matrix}1 & 1 & 1 \\ 2 & 0 & 0 \\ 0 & 3 & 0 \\ 0 & 0 & n
\end{matrix}\right]\longrightarrow
\left[\begin{matrix}1 & 1 & 1 \\ 0 & -2 & -2 \\ 0 & 3 & 0 \\ 0 & 0 & n
\end{matrix}\right]\longrightarrow
\left[\begin{matrix}1 & 1 & 1 \\ 0 & -2 & -2 \\ 0 & 1 & -2 \\ 0 & 0 & n
\end{matrix}\right]\longrightarrow
\left[\begin{matrix}1 & 0 & 3 \\ 0 & 1 & -2 \\ 0 & 0 & 6 \\ 0 & 0 & n
\end{matrix}\right]
\end{equation*}
So $\Zd(2,3,n)/\Zd'(2,3,n)$ is a cyclic group of order $k=\text{GCD}(6,n)$.  This computation also shows that $\zd_1\in \Zd'(2,3,n)$ if and only if $n\not\equiv 0 \text{ mod } 2$, that $\zd_2\in \Zd'(2,3,n)$ if and only if $n\not\equiv 0 \text{ mod } 3$, and that $\zd_3^k$ is the smallest power of $\zd_3$ in $\Zd'(2,3,n)$. In particular $\zd_3^k$ is a nontrivial elliptic element of $\Zd'(2,3,n)$.  Every element of $\Zd'(2,3,n)$ which commutes with $\zd_3^k$ must fix the fixed point of $\zd_3^k$.  It easily follows that the center of $\Zd'(2,3,n)$ is trivial, and in the same way that the center of $\Zd(2,3,n)$ is trivial.

Since the kernel of the homomorphism from $\Zc(2,3,n)$ to $\Zd(2,3,n)$ is generated by the central element $\zc_1\zc_2\zc_3$ and the center of $\Zd(2,3,n)$ is trivial, it follows that the kernel of this homomorphism is the center of $\Zc(2,3,n)$.  So $\Zc(2,3,n)$ modulo its center is isomorphic to $\Zd(2,3,n)$.  Similarly, $G_n$ modulo its center is isomorphic to $\Zd'(2,3,n)$.

Now suppose that $n\equiv\pm 1 \text{ mod } 6$.  Then $G_n$ modulo its center is isomorphic to $\Zd'(2,3,n)= \Zd(2,3,n)$.  The largest order of a torsion element in $\Zd(2,3,n)$ is $n$.  So $G_m$ and $G_n$ are distinct if $m\equiv n\equiv \pm 1 \text{ mod } 6$.  Next suppose that $n\equiv \pm 2 \text{ mod } 6$.  In this case the largest order of a torsion element in $\Zd'(2,3,n)$ is $n/2$.  So $G_m$ and $G_n$ are distinct if $m\equiv n\equiv \pm 2 \text{ mod } 6$.  The same argument is valid if $n\equiv 3 \text{ mod } 6$.  Finally suppose that $n\equiv 0 \text{ mod } 6$.  In this case neither $\zd_1$ nor $\zd_2$ are in $\Zd'(2,3,n)$.  In this case every torsion element in $\Zd'(2,3,n)$ is conjugate to a power of $\zd_3^6$, which has order $n/6$.  Again $G_m$ and $G_n$ are distinct if $m\equiv n\equiv 0 \text{ mod } 6$.
\end{proof}

\section[History]{History}\label{sec-history}

There is a large literature concerning the Fibonacci groups, the Sieradski groups, their generalizations, cyclic presentations of groups, the relationship between cyclic presentations and branched cyclic covers of manifolds, two-bridge knots, and their generalizations. We are incapable of digesting, let alone giving an adequate summary, of this work. We plead forgiveness forhaving omitted important and beautiful work and for misrepresenting work that we have not adequately studied.

\subsection{The Fibonacci groups}  John Conway told the first-named author of this paper that he created the Fibonacci group $F(5)$, with presentation
\[
\langle x_1,\ldots, x_5\mid x_1x_2=x_3, x_2x_3=x_4, x_3x_4=x_5, x_4x_5=x_1, x_5x_1=x_2\rangle
\]
and asked that his graduate students calculate its structure as an exercise to demonstrate that it is not easy to read the structure of a group from a group presentation. For example, our straightforward coset enumeration program creates 4 layers and more than 200 vertices before the coset graph collapses to its final 11 elements. Conway presented the calculation as a problem in~\cite{Con65}. The definition was immediately generalized to give the group $F(n)$. Coset enumeration showed that $F(n)$ is finite for $n < 6$ and for $n = 7$. The Cayley graph for group $F(6)$ can be constructed systematically and recognized as a $3$-dimensional infinite Euclidean group. Roger Lyndon proved, using small cancellation theory, that $F(n)$ is infinite if $n \ge 11$ (unpublished). A.~M.~Brunner~\cite{Bru74} proved that $F(8)$ and $F(10)$ are infinite. George Havas, J.~S.~Richardson, and Leon S.~Sterling showed that $F(9)$ has a quotient of order $152\cdot 5^{18}$, and, finally, M.~F.~Newman~\cite{New90} proved that $F(9)$ is infinite. Derek F.~ Holt later reported a proof by computer that $F(9)$ is automatic, from which it could be seen directly from the word-acceptor that the generators have infinite order.

At the International Congress in Helsinki (1978), Bill Thurston was advertising the problem (eventually solved by Misha Gromov) of proving that a group of polynomial growth has a nilpotent subgroup of finite index. The first-named author brought up the example of $F(6)$ as such a group.  Thurston immediately recognized the group as a branched cyclic cover of $\Sth$, branched over the figure-eight knot. And before our dinner of reindeer steaks was over, Thurston had conjectured that the even-numbered Fibonacci groups were probably also branched cyclic covers of $\Sth$, branched over the figure-eight knot. This conjecture was verified by H.~M.~Hilden, M.~T.~Lozano, and J.~M.~Montesinos-Amilibia~\cite{Hil90} and by H.~Helling, A.~C.~Kim, and J.~L.~Mennicke~\cite{Hel98}. C.~Maclachlan~\cite{Mac95} proved that, for odd $n$, the group $F(n)$ is not a fundamental group of a hyperbolic $3$-orbifold of finite volume.

\subsection{Sieradski manifolds}
The Sieradski manifolds have a similar rich history, but not one we know as well. As examples, they were introduced by A.~Sieradski in 1986~\cite{Sie86}. Sieradski used the same faceted $3$-ball that we employ, though his face pairings were different. Richard M. Thomas~\cite{Tho91b} shows that the Sieradski groups, which he calls $G(n)$, are infinite if and only if $n\ge 6$ and that $G(6)$ is metabelian. A.~Cavicchioli, F.~Hegenbarth, and A.~C.~Kim~\cite{Cav98} show that the Sieradski manifolds are branched over the trefoil knot.

\subsection{Cyclic presentations}
Cyclic presentations are particularly interesting because of their connections with branched cyclic coverings of $3$-manifolds. Fundamental results about cyclic presentations appear in the book \emph{Presentations of Groups} by
D.~L.~Johnson \cite[Chapter 16]{Joh76}. Arye Juh\'asz~\cite{Juh07} considers the question of when cyclically presented groups are finite. Andrzej Szczepa\'nski and Andrei Vesnin \cite{Szc00} ask which cyclically presented groups can be groups of hyperbolic $3$-orbifolds of finite volume and which cannot. Alberto Cavicchioli and Fulvia Spaggiari~\cite{Cav06} show that non-isomorphic cyclically presented groups can have the same polynomial.

\subsection{Dunwoody manifolds}
M.~J.~Dunwoody \cite{Dun95} managed to enumerate, with parameters, a large class of $3$-manifolds admitting Heegaard splittings with cyclic symmetry. The fundamental groups were all cyclically presented. He observed that the polynomials associated with the cyclic presentations were Alexander polynomials of knots and asked whether the spaces were in fact branched cyclic covers of $\Sth$, branched over knots or links. Alberto Cavicchioli, Friedrich Hegenbarth, and Ann Chi Kim \cite{Cav99} showed that the Dunwoody manifolds included branched covers with singularities that were torus knots of specific type. L.~Grasselli and M.~Mulazzani~\cite{Gra01} showed that Dunwoody manifolds are cyclic coverings of lens spaces branched over $(1,1)$-knots. Cavicchiolo, Beatrice Ruini, and Fulvia Spaggiari~\cite{Cav01} proved the Dunwoody conjecture that the Dunwoody manifolds are $n$-fold cyclic coverings branched over knots or links. Soo Hwan Kim and Yangkok Kim \cite{Kim04} determined the Dunwoody parameters explicity for a family of cyclically presented groups that are the $n$-fold cyclic coverings branched over certain torus knots and certain two-bridge knots. Nurullah Ankaralioglu and Huseyin Aydin~\cite{Ank08} identified certain of the Dunwoody parameters with
generalized Sieradski groups.

\subsection{Two-bridge knots}
The first general presentation about the branched cyclic coverings of the two-bridge knots seems to be that of Jerome Minkus in 1982 \cite{Min82}. A very nice presentation appears in~\cite{Cav99b} where cyclic presentations are developed that correspond to cyclically symmetric Heegaard decompositions. The
authors Alberto Cavicchioli, Beatrice Ruini, and Fulvia Spaggiari show that the polynomial of the presentation is the Alexander polynomial. They use the very clever and efficient RR descriptions of the Heegaard decompositions. They pass from the Heegaard decompositions to face-pairings and determine many of the geometric structures. Michele Mulazzani and Andrei Vesnin~\cite{Mul01} exhibit the many ways cyclic branched coverings can be viewed: polyhedral, Heegaard, Dehn surgery, colored graph constructions.

In addition to these very general presentations, there are a number of concrete special cases in the literature: \cite{Ble88, Kim98, Kim00, Kim03, Kim04, Jeo06, Jeo08, Gra09b, Tel10}.

Significant progress has been made beyond the two-bridge knots. C.~Maclachlan and A.~Reid \cite{Mac97} and  A.~Yu.~Vesnin and A.~Ch.~Kim \cite{Ves98b}  consider $2$-fold branched covers over certain $3$-braids. Alexander Mednykh and Andrei Vesnin \cite{Med95} consider $2$-fold branched covers over Turk's head links. 

Alessia Cattabriga and Michele Mulazzani \cite{Mul03, Cat03} develop strongly-cyclic branched coverings with cyclic presentations over the class of $(1,1)$ knots, which includes all of the two-bridge knots as well as many knots in lens spaces. P.~Cristofori, M. Mulazzani, and A. Vesnin \cite{Cri07} describe strongly-cyclic branched coverings of knots via $(g,1)$-decompositions. Every knot admits such a description.

\end{document}